\documentclass[12pt]{article}

\usepackage[colorlinks=true, pdfstartview=FitV, linkcolor=blue, citecolor=blue, urlcolor=blue]{hyperref}

\usepackage[OT2,OT1]{fontenc}

\usepackage{amssymb,amsmath, amscd,amsthm,bbm}
\usepackage{times, verbatim}
\usepackage{graphicx}
\usepackage[all]{xy}
\usepackage{comment}

\DeclareFontFamily{OT1}{rsfs}{}
\DeclareFontShape{OT1}{rsfs}{n}{it}{<-> rsfs10}{}
\DeclareMathAlphabet{\mathscr}{OT1}{rsfs}{n}{it}

\newtheorem{theorem}{Theorem}
\newtheorem{theoremx}{Theorem}[section]

\newtheorem{corollaryx}[theoremx]{Corollary}
\newtheorem{lemmax}[theoremx]{Lemma}
\newtheorem{propositionx}[theoremx]{Proposition}
\newtheorem{corollary}[theorem]{Corollary}
\newtheorem{lemma}[theorem]{Lemma}
\newtheorem{defn}[theorem]{Definition}
\newtheorem{proposition}[theorem]{Proposition}
\newtheorem{remark}[theorem]{Remark}
\newtheorem{notation}[theoremx]{Notation}

\def\iso{{\cong}}
\def\p{{\mathfrak p}}
\def\q{{\mathfrak q}}
\def\cF{{\mathcal F}}
\def\cS{{\mathcal S}}
\def\X{C} 
\def\RC{{\mathcal C}}
\def\A{{\mathbb A}}
\def\Cy{{\rm C}}

\def\tam #1#2{\tilde c_{\scriptscriptstyle #1/#2}}

\def\llara{\langle,\rangle}
\def\neron#1{\omega_{#1}^o}
\let\le\leqslant
\let\ge\geqslant
\def\z{y}

\def\fun {admissible}

\DeclareMathOperator{\rk}{rk}
\DeclareMathOperator{\ord}{ord}
\DeclareMathOperator{\Ind}{Ind}

\DeclareMathOperator{\End}{End}

\DeclareMathOperator{\SL}{SL}
\DeclareMathOperator{\GL}{GL}

\DeclareMathOperator{\disc}{disc}
\DeclareMathOperator{\Res}{Res}
\DeclareMathOperator{\Spec}{Spec}
\DeclareMathOperator{\Sel}{Sel}
\DeclareMathOperator{\Gal}{Gal}
\DeclareMathOperator{\Pic}{Pic}

\DeclareMathOperator{\Sym}{Sym}
\DeclareMathOperator{\Span}{Span}
\DeclareMathOperator{\Stab}{Stab}
\DeclareMathOperator{\Br}{Br}

\def\be{{\beta}}

\def\disc{{\rm disc}}

\def\dim{{\rm dim}}

\def\Div{{\rm Div}}
\def\Vol{{\rm Vol}}
\def\R{{\mathbb R}}
\def\F{{\mathbb F}}
\def\FF{{\mathcal F}}
\def\RR{{\mathcal R}}

\def\Q{{\mathbb Q}}

\def\H{{\mathcal H}}

\def\J{{\mathcal J}}
\def\C{{\mathcal C}}

\def\Z{{\mathbb Z}}
\def\P{{\mathbb P}}
\def\F{{\mathbb F}}
\def\FF{{\mathcal F}}
\def\Q{{\mathbb Q}}
\def\C{{\mathbb C}}
\def\A{{\mathbb A}}

\def\triv{{\mathbf 1}}
\def\H{{\mathcal H}}

\def\Jac{{\textrm{Jac}}}

\def\CC{{\mathcal C}}

\def\RR{{\mathcal R}}
\def\II{{\mathcal I}}

\DeclareSymbolFont{cyrletters}{OT2}{wncyr}{m}{n}
\DeclareMathSymbol{\Sha}{\mathalpha}{cyrletters}{"58}

\def\Rlk {{\rm{Res}}_{L/K}}
\def\al {\alpha}
\def\TIMES {\times}
\newcommand{\gl}[2]{{^{#2\!}#1}}
\newcommand{\w}[1]{\widetilde{#1}}
\renewcommand{\O}{{\rm O}}
\newcommand*{\longhookrightarrow}{\ensuremath{\lhook\joinrel\relbar\joinrel\rightarrow}}

\title{\!\!\!A positive proportion of locally soluble hyperelliptic curves\!\\ over $\Q$ \:\!have no point over any odd degree extension}

\author{Manjul Bhargava, Benedict H.\ Gross, and Xiaoheng Wang\\[.025in](with
  an appendix by Tim and Vladimir Dokchitser)}

\begin{document}
\maketitle


\section{Introduction}\label{sec:intro}

In this article, for any fixed genus $g\geq 1$,
we prove that a positive proportion of hyperelliptic curves over $\Q$ of genus $g$ have points over $\R$ and over $\Q_p$ for all $p$, but have no points globally over {\it any} extension of $\Q$ of odd degree.

By a hyperelliptic curve over
$\Q$, we mean a smooth, geometrically irreducible, complete curve $C$ over $\Q$ equipped with a fixed
map of degree 2 to $\P^1$ defined over~$\Q$.
Thus any hyperelliptic curve $C$ over $\Q$ of genus~$g$ can be
embedded in weighted projective space $\P(1,1,g+1)$ and
expressed by an equation of the form
\begin{equation}\label{hypereq}
C: z^2 = f(x,y) = f_0x^n+f_1x^{n-1}y+\cdots+f_ny^n
\end{equation}
where $n=2g+2$, the coefficients $f_i$ lie in $\Z$, and $f$ factors into distinct linear factors over $\bar\Q$.
Define the height $H(C)$ of~$C$~by
\begin{equation}\label{heightdef}
 H(C):= H(f):=\max\{|f_i|\}.
 \end{equation}
Then there are clearly only finitely many integral equations (\ref{hypereq}) of height less than~$X$, and we use the height to enumerate the hyperelliptic curves of a fixed genus $g$ over $\Q$.

We say that a variety over $\Q$ is {\it locally soluble} if it has a point over $\Q_\nu$ for every place $\nu$ of~$\Q$, and is {\it soluble} if it has a point over $\Q$. It is known that most hyperelliptic curves over $\Q$ of any fixed genus $g\geq 1$ when ordered by height are locally soluble (cf.\ \cite{PS2} and \cite{BCF},
where it is shown that more than 75\% of hyperelliptic curves have this property).

The purpose of this paper
is to prove the following theorem.

\begin{theorem}\label{count2}
Fix any $g\geq 1$.  Then a positive proportion of locally soluble hyperelliptic curves over~$\Q$ of~genus~$g$
have no points over {any} odd degree extension of~$\Q$.
\end{theorem}

Let $J=\Pic^0_{C/\Q}$ denote the Jacobian of $C$ over $\Q$, which is an abelian
variety of dimension~$g$. The points of~$J$ over a finite extension
$K$ of $\Q$ are the divisor classes of degree zero on $C$ that are
rational over~$K$. (When $C$ is locally soluble, we will see that
every $K$-rational divisor class on $C$ is represented by a
$K$-rational divisor.) Let $J^1=\Pic^1_{C/\Q}$ denote the principal homogeneous
space for~$J$ whose points correspond to the divisor classes of
degree one on $C$. A point $P$ on~$C$ defined over an extension field $K/\Q$
of odd degree~$k$ gives a rational point on $J^1,$ by taking the
class of the degree-one divisor that is the sum of the distinct
conjugates of $P$ minus~$(k-1)/2$ times the hyperelliptic~class~$d$
obtained by pulling back $\mathcal O(1)$ from~$\P^1$.
Thus Theorem \ref{count2} is
equivalent to the following:

\begin{theorem}\label{count3}
Fix any $g\geq 1$. For a positive proportion of locally soluble hyperelliptic curves $C$ over~$\Q$ of genus~$g$, the variety $J^1$ has no rational points.
\end{theorem}

To prove Theorems \ref{count2} and \ref{count3}, we show that for a positive proportion of locally
soluble hyperelliptic curves $C$ over $\Q$, the varieties $J$ and $J^1$ are not
isomorphic over $\Q$. To distinguish these varieties, which become isomorphic over
$\overline{\Q}$, we will study their arithmetic fundamental groups. In fact, we need only
the quotient of the arithmetic fundamental group given by two-covers.

Let $I$ be a principal homogeneous space for the abelian variety $J$. A {\it two-cover}
of $I$ is, by definition, an unramified covering $\pi: Y \rightarrow I$ by another
principal homogeneous space $Y$ for $J$ with the property that
$$\pi(y+a) = \pi(y) + 2a$$
for any $y \in Y$ and $a \in J$. The degree of any two-cover is $2^{2g}$.

The simplest example of a two-cover of $J$ is given by the multiplication-by-$2$ isogeny $J \xrightarrow{2} J$.
Another interesting two-cover of $J$ is $J^1 \xrightarrow{2} J^2 \cong J$, where the first map is
multiplication by $2$ in $\Pic_{C/\Q}$ and $J^2$ is identified with $J$ by translation by the
hyperelliptic class~$d$ of degree $2$. If $\pi: Y \rightarrow J$ is any two-cover of $J$,
then the fiber over the origin gives a principal homogeneous space $Y[2]$ for the $2$-torsion
subgroup $J[2]$, and the class of this homogeneous space in the Galois cohomology group
$H^1(\Q,J[2])$ determines the isomorphism class of the two-cover $\pi$.

The two-covers $\pi:Y \rightarrow J$ where $Y$ has points over $\Q_\nu$ for all places $\nu$
are called {\it locally soluble}. They correspond to elements in the $2$-Selmer subgroup $\Sel_2(J)$
of $H^1(\Q,J[2])$. The $2$-Selmer group is finite, and lies in an exact sequence
$$0 \rightarrow J(\Q)/2J(\Q) \rightarrow \Sel_2(J) \rightarrow \Sha_J[2] \rightarrow 0.$$
The isogeny $J \xrightarrow{2} J$ corresponds to the trivial class in the Selmer group, and the
two-cover $J^1 \xrightarrow{2} J^2 \cong J$ gives a class $W[2]$ in the Selmer group whenever
$C$ (and hence $J^1$) is locally soluble. This class turns out to be non-trivial $100\%$ of the
time, as points of $W[2]$ correspond to Weierstrass divisors $e$ of degree $1$ on $C$ with
$2e \equiv d$. These divisors $e$ correspond to odd factorizations of $f(x,y)$ over $\Q$. An {\it odd} (resp.\ {\it even}) {\it factorization} of $f(x,y)$ over $\Q$ is a factorization of the form $f(x,y)=g(x,y)h(x,y)$ where $g,h$ are odd (resp.\ even) degree binary forms that are either defined over $\Q$ or are conjugate over some quadratic extension of~$\Q$. By Hilbert's irreducibility theorem, such factorizations rarely exist. The class $W[2]$ maps to the trivial class in $\Sha_J[2]$ if and only if $J^1$ has a rational point. Hence as an immediate corollary of Theorem \ref{count3}, we obtain:

\begin{corollary}
Fix any $g\geq1$. Then a positive proportion of locally soluble hyperelliptic curves over~$\Q$ of genus~$g$ have nontrivial $2$-torsion in the Tate-Shafarevich groups of their Jacobians.
\end{corollary}

\begin{remark} \emph{Another consequence of the fact that odd and even factorizations of a binary form $f(x,y)$ over $\Q$ rarely exist is that for 100\% of all locally soluble hyperelliptic curves $C$ over $\Q$, the set $J^1(\Q)$ is either empty or infinite. Indeed, if $J^1$ has a rational point, then the class of $W[2]$ in $H^1(\Q,J[2])$ lies in the image of the group $J(\Q)/2J(\Q)$. If $f(x,y)$ has no odd or even factorization over $\Q$, then $W[2]$ is nontrivial and $J(\Q)[2]=0$. Therefore, $J(\Q)$ has positive rank and hence is infinite, and as a consequence $J^1(\Q)$ is infinite.}
\end{remark}

Similarly, we define the $2$-{\it Selmer set} $\Sel_2(J^1)$ of $J^1$ as the set of isomorphism
classes of locally soluble two-covers $\pi:Y \rightarrow J^1$. This finite set is either empty
or forms a principal homogeneous space for the finite group $\Sel_2(J)$. In fact, $\Sel_2(J^1)$ is the
set of all elements in the $4$-Selmer group $\Sel_4(J)$ which map to the class of $W[2]$ in $\Sel_2(J)$
in the first descent.

When the set $\Sel_2(J^1)$ is empty, the varieties $J^1$ and $J$ are
non-isomorphic, and distinguished by their two-covers. We will prove:

\begin{theorem}\label{counttwocover}
Fix any $g\geq 1$. For a positive proportion of locally soluble hyperelliptic curves $C$ over~$\Q$, the $2$-Selmer set $\Sel_2(J^1)$ is empty.
\end{theorem}

\begin{theorem}\label{count}
Fix any $g\geq 1$.  When all locally~soluble hyperelliptic curves $C$ over~$\Q$ of genus $g$ are ordered by height, the average size of the $2$-Selmer set $\Sel_2(J^1)$ is at most~$2$.
\end{theorem}
\noindent
We expect that the average in Theorem \ref{count} is in fact equal to $2$, and thus is independent of $g$. To prove Theorem \ref{count}, we will use the theory of pencils of quadrics to
construct and count the locally soluble two-covers of $J^1$.

Our methods also allow us to count elements, on average, in more
general 2-Selmer sets. For~$C$ a hyperelliptic curve over~$\Q$ having
hyperelliptic class $d$, and $k>0$ any odd integer, define the {\it
  $2$-Selmer set of order $k$} for $C$ to be the subset of elements of
$\Sel_2(J^1)$ that locally come from $\Q_\nu$-rational points on $J^1$
of the form $e_\nu-\frac{k-1}{2}d$, where $e_\nu$ is an effective
divisor of odd degree $k$ on $C$ over~$\Q_\nu$, for all places $\nu$.
Then we show:

\begin{theorem}\label{countk}
  Fix any odd integer $k>0$. Then the average size of the $2$-Selmer
  set of order $k$, over all locally soluble hyperelliptic curves of
  genus $g$ over~$\Q$, is strictly less than $2$ provided that $k <
  g$, and tends to $0$ as $g\to\infty$.
\end{theorem}
Theorem~\ref{countk} implies that most hyperelliptic curves of large genus have no $K$-rational points over all extensions $K$ of $\Q$ having small odd degree:

\begin{corollary}\label{corcountk}
Fix any $m>0$. Then as $g\to\infty$, a proportion approaching $100\%$ of hyperelliptic curves~$C$ of genus $g$ over~$\Q$ contain no points over all extensions of $\Q$ of odd degree $\leq m$.
\end{corollary}

Corollary \ref{corcountk} allows us to construct many smooth surfaces and varieties of higher degree, as symmetric powers of hyperelliptic curves, that fail the Hasse principle:
\begin{corollary}\label{hasse}
Fix any odd integer $k>0$.  Then as $g\to\infty$, the variety $\Sym^k(C)$ fails the Hasse principle for a proportion approaching $100\%$ of locally soluble hyperelliptic curves $C$ over~$\Q$ of genus~$g$.
\end{corollary}

One may ask what is the obstruction to the Hasse principle for the
varieties $J^1$ and $\Sym^k(C)$ occurring in
Theorem~\ref{count3} and Corollary~\ref{hasse}, respectively. In both cases, the obstruction arises from the non-existence of a locally soluble two-cover of~$J^1$.
As shown by Skorobogatov~\cite[Theorem~6.1.1]{Sk} (see also Stoll~\cite[Remark~6.5 \& Theorem~7.1]{StoInv}), using the descent theory of Colliot-Th\'el\`ene and Sansuc~\cite{CS}, this obstruction yields a case of the Brauer-Manin obstruction for both $J^1$ and $\Sym^k(C)$.
Therefore, we obtain:

\begin{theorem}\label{bm}
Fix any $g\geq 1$. For a positive proportion of locally soluble
hyperelliptic curves $C$ over~$\Q$ of genus $g$, the variety $J^1$ of
dimension $g$ has a Brauer--Manin obstruction to having a rational point.
\end{theorem}

\begin{theorem}\label{bm2}
Fix any odd integer $k>0$. As $g\to\infty$, for a density approaching $100\%$ of locally soluble
hyperelliptic curves $C$ over~$\Q$ of genus $g$, the variety $\Sym^k(C)$ of
dimension $k$ has a Brauer--Manin obstruction to having a rational point.
\end{theorem}

Recall that the index $I(C)$ of a curve $C/\Q$ is the least positive degree of a $\Q$-rational divisor $D$ on $C$. Equivalently, it is the greatest common divisor of all degrees $[K : \Q]$ of finite field
extensions $K/\Q$ such that $C$ has a $K$-rational point.  Then Theorems~\ref{count2} and \ref{count3} are also equivalent to:

\begin{theorem}\label{index}
For any $g\geq 1$, a positive proportion of locally soluble hyperelliptic curves $C$ of genus $g$ over~$\Q$ have index $2$.
\end{theorem}

We will actually prove more general versions of all of these results,
where for each $g\geq 1$ we range over {\it any} ``\fun'' congruence
family of hyperelliptic curves $C$ over $\Q$ of genus~$g$ for which $\Div^1(C)$ (but not necessarily $C$) is locally soluble;
see~Definition~\ref{def:fun} for the definition of ``\fun''.

We obtain Theorem~\ref{counttwocover} from Theorem~\ref{count} by combining it
with a result of Dokchitser and Dokchitser (see Appendix~A), which
states that a positive proportion of locally soluble hyperelliptic
curves over $\Q$ of genus $g \geq 1$ have even (or odd) $2$-Selmer
rank. Indeed, suppose that $C$ is a locally soluble hyperelliptic
curve whose $2$-Selmer set $\Sel_2(J^1)$ is nonempty. Then the
cardinality of $\Sel_2(J^1)$ is equal to the order of the finite
elementary abelian
$2$-group $\Sel_2(J)$. As we have shown earlier, for $100\%$ of locally soluble hyperelliptic curves, the group $\Sel_2(J)$ contains at least $2$ elements, namely the trivial class and the class $W[2]$. Hence the cardinality of $\Sel_2(J^1)$ is at least $2$. Moreover, if the
2-Selmer rank of the Jacobian is even, then the set $\Sel_2(J^1)$
(when nonempty) will have size at least~4. Therefore,
Theorem~\ref{count} (and Appendix~A) implies that for a positive
proportion of locally soluble hyperelliptic curves, the Selmer set
$\Sel_2(J^1)$ is empty. This proves Theorem~\ref{counttwocover}.

We prove Theorem~\ref{count} by relating the problem to a purely
algebraic one involving pencils of quadrics.  Let $A$ and $B$ be two
symmetric bilinear forms over $\Q$ in $n=2g+2$ variables, and assume
that the corresponding pencil of quadrics in $\P^{n-1}$ is
generic. Over the complex numbers, the Fano variety $F=F(A,B)$ of
common maximal isotropic subspaces of $A$ and $B$ is isomorphic to the
Jacobian $J$ of the hyperelliptic curve given by $C:z^2=\disc(Ax-By)
:= (-1)^{g+1}\det(Ax - By)$ (cf.~\cite{R}, \cite{D}, \cite{DR});
furthermore, all such pairs $(A,B)$ with the same discriminant binary form
are $\SL_n(\C)$-equivalent.

However, as shown in \cite{W}, over $\Q$ the situation is much different. Given $A$ and $B$, the Fano variety $F = F(A,B)$ might not have any rational points. In general, $F$ is a principal homogeneous space for $J$ whose class $[F]$ in $H^1(\Q,J)$ has order dividing~$4$ and satisfies $2[F]=[J^1]$; hence $F$ gives a two-cover of $J^1$ (see~\cite{W} or \S\ref{generic} for more details on the properties of the Fano variety). Moreover, given a hyperelliptic curve $C:z^2=f(x,y)$ over $\Q$ of genus $g$ (equivalently, a binary form of degree $n=2g+2$ over $\Q$ with nonzero discriminant), there might not exist {\it any} pair $(A,B)$ of symmetric bilinear forms over $\Q$
such that $f(x,y)=\disc(Ax-By)$! This raises the natural question: for which binary forms $f(x,y)$ of degree $n=2g+2$ and nonzero discriminant over $\Q$ does there exist a pair $(A,B)$ of symmetric bilinear forms in $n$ variables over $\Q$ such that $f(x,y)=\disc(Ax-By)$?

In this paper, we give a geometric answer to this question in terms of
the {generalized Jacobian}~$J_{\frak m}$ of the hyperelliptic curve
$C:z^2=f(x,y)$. Assume for simplicity that $f(x,y) = f_0x^n +
f_1x^{n-1}y + \cdots+f_n y^n$ has first coefficient $f_0 \neq 0$, so
that the curve $C$ has two distinct points $P$ and $P'$ above the
point $\infty = (1,0)$ on $\P^1$. These points are rational and
conjugate over the field $\Q(\sqrt{f_0})$. Let $\frak m = P + P'$ be
the corresponding modulus over~$\Q$ and let $C_{\frak m}$ denote the
singular curve associated to this modulus as in \cite[Ch.\ IV, \S4]{S4}. Then $C_{\frak m}$ is given by the equation $z^2=f(x,y)y^2$, and has an ordinary double point at infinity. The {\it generalized Jacobian} of $C$ associated to the modulus $\frak m$, denoted by $J_{\frak m}=J_{\frak m}(C)$, is the connected component of the identity of $\Pic_{C_{\frak m}/\Q}/\Z\cdot d$,
while $J^1_{\frak m}=J^1_{\frak m}(C)$ denotes the nonidentity component; here $d$ denotes the hyperelliptic class of $C_{\frak m}$ in $\Pic^2_{C_{\frak m}/\Q}(\Q)$
obtained by
pulling back $\mathcal O(1)$ from~$\P^1$.
We prove:

\begin{theorem}\label{genJ}
Let $f(x,y)$ denote a binary form of even degree $n=2g+2$ over $\Q$, with nonzero discriminant and nonzero first coefficient.  Then there exists a pair $(A,B)$ of symmetric bilinear forms over $\Q$ in $n$ variables satisfying $f(x,y)=\disc(Ax-By)$ if and only if there exists a two-cover of homogeneous spaces $F_{\frak m} \rightarrow J_{\frak m}^1$ for $J_{\frak m}$ over $\Q$, or equivalently, if and only if the class of the homogeneous space $J_{\frak m}^1$ is divisible by $2$ in the group $H^1(\Q ,J_{\frak m})$.
\end{theorem}
See Theorem~\ref{thm:with8equiv} for a number of other equivalent conditions for the existence of $A$ and $B$ satisfying $f(x,y)=\disc(Ax-By)$.
It is of significance that the singular curve~$C_{\frak m}$ and the generalized Jacobian~$J_{\frak m}$ appear in Theorem~\ref{genJ}.  The generalized Jacobians appeared in \cite{PSh} for the purpose of doing 2-descent on the Jacobians of hyperelliptic curves with no rational Weierstrass point. As noted in~\cite[Footnote~2]{PSh}, in this case it is not always enough to study only unramified covers of $C$; one needs also covers of $C$ unramified away from the points above some fixed point on~$\P^1$.

The group $\SL_n(\Q)$ acts on the space $\Q^2\otimes\Sym_2\Q^n$ of pairs $(A,B)$ of symmetric bilinear forms on an $n$-dimensional vector space, and $\mu_2\subset \SL_n$ acts trivially since $n=2g+2$ is even.  The connection with Theorem~\ref{count} arises from the fact that we may parametrize elements of $\Sel_2(J^1)$ by certain orbits for the action of the group $(\SL_n/\mu_2)(\Q)$ on the space $\Q^2\otimes\Sym_2\Q^n$.
We say that an element $(A,B)\in\Q^2\otimes\Sym_2\Q^n$, or its $(\SL_n/\mu_2)(\Q)$-orbit, is {\it locally soluble} if the associated Fano variety $F(A,B)$ has a point locally over every place of $\Q$. Then we prove the following bijection:

\begin{theorem}\label{2Selpar}
Let $f(x,y)$ denote a binary form of even degree $n=2g+2$ over $\Q$ such that the hyperelliptic curve $C:z^2=f(x,y)$ is locally soluble. Then the $(\SL_n/\mu_2)(\Q)$-orbits of locally soluble pairs $(A,B)$ of symmetric bilinear forms in $n$ variables over $\Q$ such that $f(x,y)=\disc(Ax-By)$ are in bijection with the elements of the $2$-Selmer set $\Sel_2(J^1)$.
\end{theorem}

To obtain Theorem~\ref{count}, we require a version of Theorem~\ref{2Selpar} for integral orbits.
Let $\Z^2\otimes\Sym_2\Z^n$ denote the space of pairs $(A,B)$ of $n\times n$ symmetric bilinear forms over $\Z$.   
Then we prove the following theorem on integral representatives:

\begin{theorem}\label{2SelparZ}
There exists a positive integer $\kappa$ depending only on $n$ such that, for any integral binary form $f(x,y)$ of even degree $n=2g+2$
with $C:z^2=f(x,y)$ locally soluble over $\Q$, every $(\SL_n/\mu_2)(\Q)$-orbit of locally soluble pairs $(A,B)\in\Q^2\otimes\Sym_2\Q^n$ such that $\disc(Ax-By)=\kappa^2f(x,y)$ contains an element in $\Z^2\otimes\Sym_2\Z^n$. In other words, the $(\SL_n/\mu_2)(\Q)$-equivalence classes of locally soluble pairs $(A,B)\in \Z^2\otimes\Sym_2\Z^n$ such that $\disc(Ax-By)=\kappa^2f(x,y)$ are in bijection with the elements of $\Sel_2(J^1)$.
\end{theorem}
We will prove Theorem \ref{2SelparZ} for $\kappa = 4$ but we expect
this can be improved.
We use Theorem~\ref{2SelparZ}, together with the results of \cite{B} giving the number of $\SL_n(\Z)$-orbits on $\Z^2\otimes\Sym_2\Z^n$ having bounded height, and a sieve, to deduce Theorem~\ref{count}.

We note that the emptiness of $J^1(\Q)$
for hyperelliptic curves $C$ over $\Q$ has been demonstrated previously for certain special
algebraic
families. In \cite{CTP}, Colliot-Th\'el\`ene and Poonen
constructed one-parameter algebraic families of curves $C=C_t$ of genus $1$ and genus $2$ for which the varieties $J^1$ have a Brauer-Manin obstruction to having a rational point for all $t\in\Q$. (We note that the family of genus $2$ curves
considered in~\cite{CTP} consists of hyperelliptic curves $C$ over $\Q$ with locally soluble $J^1(C)$ but not locally soluble $\Div^1(C).$) For arbitrary genus $g\geq6$ with $4\nmid g$, Dong Quan~\cite{NDQ} constructed such one-parameter algebraic families of locally soluble hyperelliptic curves $C=C_t$ having empty $J^1(\Q)$ for every $t\in\Q$.

This paper is organized as follows.  In Section~\ref{dedorbits}, we
introduce the key representation $2\otimes\Sym_2(n)$ of $\SL_n$ on
pairs of symmetric bilinear forms that we will use to study the
arithmetic of hyperelliptic curves. We adapt the results of
Wood~\cite{Wood1} to study the orbits of this representation over a general Dedekind domain $D$ whose characteristic is not equal to $2$.  In Section~\ref{hyp}, we introduce hyperelliptic curves and some of the relevant properties of their generalized Jacobians.
In Section~\ref{generic}, we then relate hyperelliptic curves to generic pencils of quadrics over a field $K$ of characteristic not equal to 2, and we review the results that we will need from \cite{W}.  In Section~\ref{regular}, we then study {\it regular} pencils of quadrics, which allows us to determine which binary $n$-ic forms over $K$ arise as the discriminant of a pencil of quadrics over $K$; in particular, we prove Theorem~\ref{genJ}.

In Section~\ref{soluble}, we describe how the {\it $K$-soluble} orbits (i.e., orbits of those $(A,B)$ over $K$ such that $F(A,B)$ has a $K$-rational point), having associated hyperelliptic curve $C$ over $K$, are parametrized by elements of the set $J^1(K)/2J(K)$. We study the orbits over some arithmetic fields in more detail in Section~\ref{arithmetic} and then we focus on global fields and discuss {\it locally soluble} orbits in Section~\ref{sec:global}. We show that the locally soluble orbits over $\Q$, having associated hyperelliptic curve $C$ over $\Q$ are parametrized by the elements of the finite set $\Sel_2(J^1)$, proving Theorem~\ref{2Selpar}. The existence of integral orbits (Theorem \ref{2SelparZ}) is demonstrated in Section~\ref{integral}.  We then discuss the counting results from~\cite{B} that we need in Section~\ref{counting}, and discuss the details of the required sieve in Section~\ref{sec:sieve}.  Finally, we complete the proofs of Theorems~\ref{count} and \ref{countk} in the final Section~\ref{mainproofs}.

\section{Orbits of pairs of symmetric bilinear forms over a Dedekind domain}\label{dedorbits}

In this section, we study the orbits of our key representation $2\otimes\Sym_2(n)$ over a Dedekind domain $D$.   In~later sections, we will specialize to the case when $D$ is a field, $\Z_p$ or $\Z$. We will also relate these results on orbits to the arithmetic of hyperelliptic curves.

Let $K$ denote the quotient field of $D$. We assume throughout this paper that the characteristic of $K$ is not equal to $2$. Let $n \geq 2$ be an integer. The group $\SL_n(D)$ acts on the $D$-module of pairs $(A,B)$ of symmetric bilinear forms on a free $D$-module $W$ of rank $n$. After a choice of basis for $W$, this is the representation $D^2 \otimes \Sym_2 D^n = \Sym_2 D^n \oplus \Sym_2 D^n$.

The coefficients of the binary $n$-ic form
$$f(x,y) = \disc(xA - yB) := (-1)^{n(n-1)/2} \det(xA-yB) = f_0 x^n + f_1 x^{n-1}y + \cdots + f_n y^n,$$
which we call the {\it invariant binary $n$-ic form} of the element $(A,B)\in D^2 \otimes \Sym_2 D^n$, give $n+1$ polynomial invariants of degree $n$ which freely generate the ring of polynomial invariants over $D$.
We also have the invariant {\it discriminant} polynomial $\Delta(f) = \Delta (f_0,f_1,\ldots, f_n)$ given by the discriminant of the binary form $f$, which has degree $2n(n-1)$ in the entries of $A$ and $B$.

In Wood's work~\cite{Wood1}, the orbits of $\SL_n^\pm(T)=\{g\in\GL_n(T):\det(g)=\pm1\}$ on $T^2\otimes\Sym_2T^n$ were classified for general rings (and in fact even for general base schemes)  $T$ in terms of ideal classes of rings of rank $n$ over $T$.  In this section, we translate these results into a form that we will use later on, in the important special case where $T=D$ is a Dedekind domain with quotient field~$K$.  In particular, we will need to use the actions by the groups $\SL_n(D)$ and in the case $n$ is even, the group $(\SL_n/\mu_2)(D)$ rather than $\SL_n^\pm(D)$. This causes some key changes in the parametrization data and will indeed be important for us when we make the connection with hyperelliptic curves.

Let us assume that $f_0\neq 0$ and write $f(x,1) = f_0 g(x)$, where $g(x)$ has coefficients in the quotient field $K$ and has $n$ distinct roots in a separable closure $K^s$ of $K$. Let $L = L_f:= K[x]/g(x)$ be the corresponding \'etale algebra of rank $n$ over $K$, and let $\theta$ be the image of $x$ in the algebra $L$. Then $g(\theta) = 0$ in $L$. Let $g'(x)$ be the derivative of $g(x)$ in $K[x]$; since $g(x)$ is separable, the value $g'(\theta)$ must be an invertible element of $L$. We define $f'(\theta) = f_0 g'(\theta)$ in $L^\times$.

For $k = 1,2,\ldots,n-1$, define the integral elements
$$\zeta_k = f_0 \theta^k + f_1 \theta^{k-1} + \cdots + f_{k-1} \theta$$
in $L$, and let $R=R_f$ be the free $D$-submodule of $L$ having $D$-basis $\{ 1,\zeta_1,\zeta_2, \ldots,\zeta_{n-1}\}$. For $k=0,1,\ldots, n-1$, let $I(k)$ be the free $D$-submodule of $L$ with basis $\{ 1, \theta, \theta^2,\ldots, \theta^{k}, \zeta_{k+1}, \ldots, \zeta_{n-1}\}$. Then $I(k)=I(1)^k$, and $I(0) = R \subset I(1) \subset \cdots \subset I(n-1)$. Note that $I(n - 1)$ has the power basis $\{ 1, \theta, \theta^2 ,\ldots ,\theta^{n-1} \}$, but that the elements of $I(n-1)$ need not be integral when $f_0$ is not a unit in $D$.

A remarkable fact (cf.~\cite{BM},~\cite[Proposition 1.1]{Nakagawa},~\cite[\S2.1]{Wood}) is that $R$ is a $D$-order in $L$ of discriminant $\Delta(f)$, and the free $D$-modules $I(k)$ are all fractional ideals of $R$. The fractional ideal
$(1/f'(\theta))I(n-2)$ is the dual of $R$ under the trace pairing on $L$, and the fractional ideal $I(n-3)$ will play a crucial role in the parametrization of orbits in our representation.

We then have the following translation of \cite[Theorem~1.3]{Wood1} in the case of the action of $\SL_n(D)$ on $D^2\otimes\Sym_2D^n$, where $D$ is a Dedekind domain:

\begin{theorem}\label{orbit}
Assume that $f(x,y)$ is a binary form of degree $n$ over $D$ with
$\Delta(f) \neq 0$ and $f_0 \neq 0$. Then there is a bijection {\em
  (}to be described below{\em )} between
orbits for $\SL_n(D)$ on $D^2 \otimes \Sym_2 D^n$ with invariant form
$f$ and
equivalence classes of triples $(I,\alpha,s)$, where
$I$ is a fractional ideal for~$R$, $\alpha \in L^\times$, and $s \in
K^\times$, satisfying the relations $I^2 \subset \alpha I(n-3)$,
$N(I)$ is the principal fractional ideal $sD$ in $K$, and $N(\alpha) = s^2 f_0^{n-3}$
in $K^\times$.
The triple $(I,\alpha, s)$ is equivalent to the triple $(cI,c^2\alpha, N(c)s)$ for any $c \in
L^\times$. The stabilizer of a triple $(I,\alpha,s)$ is $S^\times[2]_{N=1}$ where $S=\End_R(I)\subset~L$.
\end{theorem}

{}From a triple $(I,\alpha, s)$, we construct an orbit as follows. Since $N(I)$ is the principal $D$-ideal $sD$, the projective $D$-module $I$ of rank $n$ is free. Since $I^2
\subset \alpha I(n-3)$, we obtain two symmetric bilinear forms on the free module $I$ by defining $\langle \lambda, \mu \rangle_A$ and $\langle \lambda, \mu \rangle_B$ as
the respective coefficients of $\zeta_{n-1}$ and $\zeta_{n-2}$ in the basis expansion of the product $\lambda\mu/\alpha$ in $I(n-3)$. We obtain an $\SL_n(D)$-orbit of two symmetric $ n \times n$ matrices $(A,B)$ over $D$ by taking the Gram matrices of these forms with respect to any ordered basis of $I$ that gives rise to the basis element $s(1 \wedge \zeta_1 \wedge \zeta_2 \wedge \ldots \wedge \zeta_{n-1})$ of the top exterior power of $I$ over $D$. This normalization deals with the difference between $\SL_n(D)$- and $\GL_n(D)$-orbits.
The stabilizer statement follows because elements in $S^\times[2]_{N=1}$ are precisely the elements of $L^\times_{N=1}$ that preserve the map $\frac{1}{\alpha}\times:I\times I\rightarrow I(n-3).$

Conversely, given an element $(A,B)\in D^2\otimes\Sym_2D^n$, we construct the ring $R=R_f$ from $f$ as described above, where $f(x,y)=\disc(xA-yB)$.  The $R$-module $I$ is then
constructed by letting $\theta \in L$ act on
$K^n$ by the matrix $A^{-1}B$. Then $\zeta_1 = f_0\theta \in R$ preserves the lattice $D^n$. Similarly, formulas for the action of each $\zeta_i \in R$ on
$D^n$, in terms of integral polynomials in the entries of $A$ and $B$, can
be worked out when $A$ is assumed to be invertible; these same formulas
can then be used to show that $D^n$ is an $R$-module, even when $A$ is not invertible.  See \cite[\S3.1]{Wood1} for the
details.

\bigskip

When $n = 2m$ is even, the larger group $(\SL_n/\mu_2) (D)$ acts on the representation $D^2 \otimes \Sym_2D^n$, and distinct orbits for the subgroup $\SL_n(D)/\mu_2(D)$ may become identified as a single orbit for the larger group. Since a projective module of rank $n$ over $D$ whose top exterior power is a free module is itself free of rank $n$ by \cite[Theorem 1.6]{MK}, we have $H^1(D, \SL_n) = 1$ and hence an exact sequence of groups
$$ 1 \rightarrow \SL_n(D)/\mu_2(D) \rightarrow (\SL_n/\mu_2) (D) \rightarrow H^1(D, \mu_2) \rightarrow 1.$$
By Kummer theory, the quotient group $H^1(D,\mu_2)$ lies in an exact sequence
$$1 \rightarrow D^\times/D^{\times2} \rightarrow H^1(D, \mu_2) \rightarrow \Pic(D)[2] \rightarrow 1.$$
The image of the group $H^1(D,\mu_2)$ in $H^1(K,\mu_2) = K^\times/K^{\times2}$ is the subgroup $K^{\times(2)}/K^{\times2}$ of elements $t$ such that the principal ideal $tD = M^2$
is a square, and the map
to $\Pic(D)[2]$ is given by mapping such an element $t$ to the class of $M$. The action of $t$ on a triple $(I,\alpha,s)$ with invariant form $f$ is given by
$$t\cdot (I,\alpha, s) = (M I,t\alpha, t^{n/2}s).$$ Along with the action of $\SL_n(D)$ on such triples, this gives an action of $(\SL_n/\mu_2) (D)$ on these triples. The equivalence classes of triples under this action of $(\SL_n/\mu_2) (D)$ give the orbits of $(\SL_n/\mu_2) (D)$ with invariant form $f$. The stabilizer of the triple $(I,\alpha,s)$ contains the finite group
$S^\times[2]_{N=1}/ D^\times[2]$ where $S=\End_R(I)\subset L$, since that is the image of the stabilizer from $\SL_n(D)$.

\begin{theorem}\label{orbit2}
Assume that $f(x,y)$ is a binary form of even degree $n$ over $D$ with
$\Delta(f) \neq 0$ and $f_0 \neq 0$. Then there is a bijection between
orbits for $(\SL_n/\mu_2)(D)$ on $D^2 \otimes \Sym_2 D^n$ with
invariant form $f$ and
equivalence classes of triples $(I,\alpha,s)$, where
$I$ is a fractional ideal for~$R$, $\alpha \in L^\times$, and $s \in K^\times$,
satisfying the relations $I^2 \subset \alpha I(n-3)$,
$N(I)$ is the principal fractional ideal $sD$ in $K$, and $N(\alpha) = s^2 f_0^{n-3}$
in $K^\times$.
The triple $(I,\alpha, s)$ is equivalent to the triple $(cMI,c^2t\alpha, N(c)t^{n/2}s)$ for any $c \in
L^\times$ and $t\in K^{\times(2)}$, where $tD=M^2$. The stabilizer of the triple $(I,\alpha,s)$ is an elementary abelian $2$-group which contains
$S^\times[2]_{N=1}/ D^\times[2]$ where $S=\End_R(I)\subset L$.

\end{theorem}

\begin{remark}\label{rmk:maximal} \emph{We can simplify the statement
    of Theorem \ref{orbit2} when the domain $D$ is a principal ideal domain and every
    fractional ideal for the $D$-order $R$ is principal. In
    that case, the fractional ideal $I$ of $R$ is completely
    determined by the pair $(\alpha, s)$ and the identities $I^2
    \subset (\alpha) I(n-3)$, $N(I) = (s)$, and $N(\alpha) = s^2
    f_0^{n-3}$. Indeed, together these force $I^2 = (\alpha)
    I(n-3)$. There is a bijection
from the set of equivalence classes of $\alpha$ to the set $(R^\times/R^{\times2}D^\times)_{N=f_0}.$ Moreover, we have $S=\End_R(I)=R$  and $K^{\times(2)}=D^\times K^{\times2}$. An element $t\in K^{\times(2)}/K^\times$ preserves an $\SL_n(D)$-orbit if and only if $t=c^2\in R^{\times2}$ for some $c\in R^\times$ with $N(c)=t^{n/2}.$ Note if $t=c^2$, then $N(c)=(-t)^{n/2}$. Hence the stabilizer in $(\SL_n/\mu_2) (D)$ of a triple $(I,\alpha,s)$ equals $(R^\times[2])_{N=1}/D^\times[2]$ if $n\equiv 2\pmod{4}$ and fits into the exact sequence}
\begin{equation}\label{eq:maximalstab}
1\rightarrow (R^\times[2])_{N=1}/D^\times[2]\rightarrow \Stab_{(\SL_n/\mu_2)(D)}(I,\alpha,s)\rightarrow (R^{\times2}\cap D^\times)/D^{\times2}\rightarrow 1,
\end{equation}
\emph{when} $n\equiv 0\pmod{4}$. \emph{When $L$ is not an algebra over a quadratic extension of $K$, the quotient \linebreak $(R^{\times2}\cap D^\times)/D^{\times2}$ is trivial.}
\end{remark}
In particular, when $D=K$ is a field, we recover \cite[Theorems 7 and 8]{AITII}. These versions of Theorems~\ref{orbit} and \ref{orbit2} over a field $K$ will also be important in the sequel. For convenience, we restate them below.

\begin{corollary}\label{cor:orbitSL}
Assume that $f(x,y)$ is a binary form of degree $n$ over~$K$ with
$\Delta(f) \neq 0$ and $f_0 \neq 0$. Then there is a bijection between
orbits for $\SL_n(K)$ on $K^2 \otimes \Sym_2 K^n$ with invariant form
$f$ and
equivalence classes of pairs $(\alpha,s)$, where $\alpha\in L^\times$ and $s\in K^\times$, satisfying $N(\alpha) = s^2 f_0^{n-3}$
in $K^\times$.
The pair $(\alpha, s)$ is equivalent to the pair $(c^2\alpha, N(c)s)$ for any $c \in
L^\times$.
The stabilizer of the orbit corresponding to a pair $(\alpha,s)$ is the finite commutative group scheme $(\Res_{L/K}\mu_2)_{N=1}$ over~$K$.
\end{corollary}
It follows from Corollary \ref{cor:orbitSL} that the set of $\SL_n(K)$-orbits is either in bijection with or has a 2-to-1 map to $(L^\times/L^{\times2})_{N=f_0}$, in accordance with whether $f(x,y)$ has an odd degree factor over $K$ or not, respectively. Indeed, the pair $(\alpha,s)$ is equivalent to the pair $(\alpha, -s)$ if and only if there is an element $c \in L^{\times}$ with $c^2 = 1$ and $N(c) = -1$. The stabilizers correspond to the $K$-rational even degree factors of $f(x,y)$.

\begin{corollary}\label{cor:orbitSLmu2}
Assume that $f(x,y)$ is a binary form of even degree $n$ over $K$ with
$\Delta(f) \neq 0$ and $f_0 \neq 0$. Then there is a bijection between
orbits of $(\SL_n/\mu_2)(K)$ on $K^2 \otimes \Sym_2 K^n$ with
invariant form $f$ and
equivalence classes of pairs $(\alpha,s)$ where $\alpha\in L^\times$ and $s\in K^\times$ satisfying $N(\alpha) = s^2 f_0^{n-3}$
in $K^\times$.
The pair $(\alpha, s)$ is equivalent to the pair $(c^2t\alpha, N(c)t^{n/2}s)$ for any $c \in
L^\times$ 
and $t\in K^{\times(2)}=K^\times$. The stabilizer of the orbit corresponding to a pair $(\alpha,s)$ is the finite commutative group scheme $(\Res_{L/K}\mu_2)_{N=1}/\mu_2$ over~$K$.
\end{corollary}
It follows from Corollary \ref{cor:orbitSLmu2} that the set of $(\SL_n/\mu_2)(K)$-orbits is either in bijection with or has a 2-to-1 map to $(L^\times/(L^{\times2}K^\times))_{N=f_0}$, in accordance with whether $f(x,y)$ has an odd factorization over $K$ or not, respectively. Here an {\it odd factorization} of $f(x,y)$ over $K$ is a factorization of the form $f(x,y)=g(x,y){h(x,y)}$, where $g$ and $h$ are odd degree binary forms that are either $K$-rational or are conjugate over some quadratic extension of $K$. Meanwhile, the elements of the stabilizer correspond to even factorizations of $f(x,y)$. When $n$ is congruent to $2$ modulo~$4$, an even factorization of $f(x,y)$ must be of the form $g(x,y)h(x,y)$ where both $g$ and $h$ are $K$-rational even degree binary forms. In other words, they already appear in the stabilizers in $\SL_n(K)$. When $n$ is congruent to $0$ modulo~$4$, $f(x,y)$ can have even factorizations into conjugate binary forms over some quadratic extensions $K'/K$. The image of a stabilizer element corresponding to such a factorization in $(L^{\times2}\cap K^\times)/K^{\times2}$ is the class corresponding to the quadratic extension $K'$.

\section{Hyperelliptic curves, divisor classes, and generalized Jacobians}\label{hyp}

Assume from now on that $n \geq 2$ is even and write $n = 2g + 2$. Fix a field $K$ of characteristic not 2. In order to interpret the orbits for $\SL_n(K)$ and $(\SL_n/\mu_2)(K)$ having a fixed invariant binary form, we first review some of the arithmetic and geometry of hyperelliptic curves of genus $g$ over $K$. As in \cite{GHanoi}, we define a hyperelliptic curve over $K$ as a smooth, projective curve over $K$ with a $2$-to-$1$ map to the projective line over $K$, although we now treat the general case (without assuming any fixed $K$-rational points at infinity.)

Let $f(x,y) = f_0x^{2g+2} + f_1x^{2g+1}y + \cdots + f_{2g+2}y^{2g+2}$ be a binary form of degree $2g+2$ over $K$, with $\Delta \neq 0$ and $f_0 \neq 0$. We associate to $f(x,y)$ the hyperelliptic curve $C$ over $K$ with equation
$$z^2 = f(x,y).$$
This defines a smooth curve of genus $g$, as a hypersurface of degree $2g+2$ in the weighted projective plane $\mathbb P(1,1,g+1)$. The weighted projective plane embeds as a surface in $\mathbb P^{g+2}$ via the map $(x,y,z) \rightarrow (x^{g+1},x^gy,\ldots,y^{g+1},z)$. The image is a cone over the rational normal curve in $\mathbb P^{g+1}$, which has a singularity at the vertex $(0,0, \ldots, 1)$ when $g \geq 1$. The curve $C$ is the intersection of this surface with a quadric hypersurface that does not pass through the vertex of the cone. Finally, the linear series on $C$ of projective dimension $g+2$ and degree $2g+2$ that gives this embedding is the sum of the all the Weierstrass points (i.e., points with $z = 0$).

There are two points $P = (1,0, z_0)$ and $P' = (1,0,-z_0)$ at infinity, where $z_0^2 = f_0$. If $f_0$ is a square in $K^\times$, then these points are rational over $K$. If not, then they are  rational over the quadratic extension $K' = K(\sqrt{f_0})$.  Let $w$ be the rational
function $z/y^{g+1}$ on $C$, and let $t$ be the rational function $x/y$ on $C$. Both are regular outside of the two points $P$ and $P'$ with $y = 0$, where they have poles of order $g+1$ and $1$ respectively. The field of rational functions on $C$ is given by $K(C) = K(t,w)$, with $w^2 = f(t,1) = f_0t^{2g+2} + f_1t^{2g+1}+\cdots + f_{2g+2}$, and the subring of functions that are regular outside of $P$ and $P'$ is $K[t,w] = K[t,\sqrt{f(t,1)}]$ \cite{GHanoi}.

Let $\frak m$ be the modulus $\frak m = P + P'$ on $C$ and let $C_{\frak m}$ be the singular curve constructed from $C$ and this modulus in \cite[Ch.\ IV, no.\ 4]{S4}. Then $C_{\frak m}$ has equation
$$z^2 = f(x,y)y^2$$
of degree $2g+4$ in $\mathbb P(1,1,g+2)$. This defines a singular, projective curve of arithmetic genus $g+1$ whose normalization is $C$. There is now a single point $Q = (1,0,0)$ at infinity, which is an ordinary double point whose tangents are rational over the quadratic extension field $K'$.

Let $\Pic_{C/K}$ and $\Pic_{C_{\frak m}/K}$ denote the Picard functors of the projective curves $C$ and $C_{\frak m}$ respectively. These are represented by commutative group schemes over $K$, whose component groups are both isomorphic to $\mathbb Z$. Let $K^s$ be a fixed separable closure of $K$ and let $E$ be any extension of $K$ contained in $K^s$. The $E$-rational points of $\Pic_{C/K}$ correspond bijectively to the divisor classes on $C$ over the separable closure $K^s$ that are fixed by the Galois group $\Gal(K^s/E)$. When the curve $C$ has no $E$-rational points, an $E$-rational divisor class on $C$ may not be represented by an $E$-rational divisor. The subgroup of classes in $\Pic_{C/K}(E)$ that are represented by $E$-rational divisors is just the image of $\Pic(C/E) = H^1(C/E,\mathbb G_m)$ in $H^0(E, H^1(C/K^s, \mathbb G_m))$, under the map induced by the spectral sequence for the morphism $C/E \rightarrow \Spec E$. From this spectral sequence, we also obtain an injection from the quotient group to the Brauer group of $K$ (cf.\ \cite[\S2.3]{Sk}, \cite[Ch.~8]{BLR}):
$$\Pic_{C/K}(K)/\Pic(C/K) \rightarrow H^2(K, \mathbb G_m) = \Br(K).$$
Since $C$ has a rational point over the quadratic extension $K' = K(\sqrt{f_0})$, the image of this injection is contained in the subgroup $\Br(K'/K) = K^\times/N(K'^\TIMES)$. Every class in $\Br(K'/K)$ corresponds to a quaternion algebra $D$ over $K$ that is split by $K'$, or equivalently, to a curve of genus zero over $K$ with two conjugate points rational over $K'$.

\begin{proposition} \label{pic}
If a hyperelliptic curve $C$ over $K$ has a rational divisor of odd degree, or equivalently a rational point over an extension of $K$ of odd degree, then every $K$-rational divisor class is represented by a $K$-rational divisor. If $K$ is a global field and $\Div^1(C)$ is locally soluble, then every $K$-rational divisor class is represented by a $K$-rational divisor.
\end{proposition}
Indeed, a quaternion algebra split by an odd degree extension of $K$ is already split over $K$. Similarly, a quaternion algebra over a global field that splits locally everywhere is split globally.

\medskip

The distinction between $K$-rational divisor classes and $K$-rational divisors does not arise for the curve $C_{\frak m}$, which always have the $K$-rational singular point $Q$. Hence the points of $\Pic_{C_{\frak m}/K}$ over $E$ correspond to the classes of divisors that are rational over $E$ and are prime to $\frak m$, modulo the divisors of functions with $f \equiv 1$ modulo $\frak m$.
We have an exact sequence of smooth group schemes over $K:$
\begin{equation}\label{eq:PicCm}
0 \rightarrow T \rightarrow \Pic_{C_{\frak m}/K} \rightarrow \Pic_{C/K} \rightarrow 0,
\end{equation}
where $T$ is the one-dimensional torus that is split by $K'$.
Taking the long exact sequence in Galois cohomology, and noting that the image of $\Pic_{C_{\frak m}/K}(K)$ in $\Pic_{C/K}(K)$ is precisely the subgroup $\Pic(C/K) = H^1(C/K,\mathbb G_m)$ represented by $K$-rational divisors, we recover the
injection $$\Pic_{C/K}(K)/\Pic(C/K) \rightarrow H^1(K,T) = K^\times/N(K'^\TIMES) = \Br(K'/K).$$ To see this geometrically, note that the fiber over a $K$-rational point $P$ of $\Pic_{C/K}$ is a principal homogeneous space for $T$ over $K$, which is a curve of genus zero with two conjugate points over $K'$ removed. This curve of genus zero determines the image of $P$ in $\Br(K'/K)$.

The connected components of the identity of the Picard schemes $J = \Pic^0_{C/K}$ and $J_{\frak m} = \Pic^0_{C_{\frak m}/K}$ are the Jacobian and generalized Jacobian of \cite[Ch.\ V]{S4}. They correspond to the divisor classes of degree zero on these curves. The exact sequence in \eqref{eq:PicCm} restricts to the following exact sequence \cite[Ch.\ V, \S3]{S4}:
\begin{equation}\label{eq:TJmJ}
0 \rightarrow T \rightarrow J_{\frak m} \rightarrow J \rightarrow 0.
\end{equation}

There is a line bundle of degree $2$ on $C_{\frak m}$ (and hence on $C$) which is the pull-back of the line bundle $\mathcal O(1)$ from the projective line under the map $(x,y,z) \to (x,y)$. This is represented by the $K$-rational divisor $d = (R) + (R')$ prime to $\frak m$ consisting of the two points above a point $(x_0,y_0)$ on the projective line, whenever $y_0$ is nonzero. The quotient groups
$\Pic_{C/K}/\mathbb Z \cdot d = J \sqcup J^1$ and $\Pic_{C_{\frak m}/K}/\mathbb Z \cdot d = J_{\frak m} \sqcup J_{\frak m}^1$ both have two connected components, represented by the divisor classes of degree $0$ and~$1$. There are morphisms
\begin{eqnarray*}
C &\longrightarrow& J^1\\
C - \{P,P'\} &\longrightarrow& J_{\frak m}^1
\end{eqnarray*}
defined over $K$, which take a point to the corresponding divisor class of degree $1$ \cite[Ch V, \S4]{S4}.

\begin{proposition} \label{prop:torsion} Let $f(x,y)=f_0x^{2g+2}+f_1x^{2g+1}y+\cdots+f_{2g+2}y^{2g+2}$ be a binary form with nonzero discriminant and nonzero $f_0$. Let $C:z^2=f(x,y)$ and $C_{\frak m}:z^2=f(x,y)y^2$ denote the associated hyperelliptic curve and singular curve with Jacobian $J$ and generalized Jacobian $J_{\frak m}$. Let $L=K[x]/f(x,1)$ denote the corresponding \'{e}tale algebra of rank $2g+2$. Then:
\begin{enumerate}
\item[$1.$] The $2$-torsion subgroup $J_{\frak m}[2]$ of $J_{\frak m}$ is isomorphic to the group scheme  \,$(\Rlk\mu_2)_{N=1}$.\, Its $K$-rational points correspond to the even degree factors of $f(x,y)$ over $K$.
\item[$2.$] The $2$-torsion subgroup $\,J[2]\,$ of $\,J\,$ is isomorphic to the group scheme $(\Rlk\mu_2)_{N=1}/\mu_2$.
Its $K$-rational points correspond to the even factorizations of $f(x,y)$ over $K$.
\item[$3.$] The $2$-torsion $W_{\frak m}[2]$ in the component $J^1_{\frak m}$ of $\Pic_{C_{\frak m}/K}/\mathbb Z \cdot d = J_{\frak m} \sqcup J^1_{\frak m}$, is a torsor for $J_{\frak m}[2]$~whose $K$-rational points correspond to the odd degree factors of $f(x,y)$ over $K$.
\item[$4.$] The $2$-torsion $W[2]$ in the component $J^1$ of $\Pic_{C/K}/\mathbb Z \cdot d = J \sqcup J^1$ is a torsor for $J[2]$ whose $K$-rational points correspond to the odd factorizations of $f(x,y)$ over $K$.
\end{enumerate}
\end{proposition}
\noindent Here an odd (resp.\ even) factorization of $f(x,y)$ over $K$ is a factorization of the form $f=gh$, where $g$ and $h$ are odd (resp.\ even) degree binary forms that are either defined over $K$ or are conjugate over some quadratic extension of $K$. Note that giving a factor of $f(x,y)$ is the same as giving a subset of Weierstrass points---hence the choice of the letter ``$W$'' in $W[2]$ and $W_{\frak m}[2]$.

\medskip

\begin{proof} To prove the proposition, we observe that the $2$-torsion points of $J_{\frak m}$ over the separable closure $K^s$ are represented by the classes of divisors of the form $(P_1) + (P_2) + \cdots + (P_{2m}) - md$, where
each $P_i = (x_i,1,0)$ comes from a distinct root $x_i$ of $f(x,1)$ \cite[\S4]{G}. Hence the points of $J_{\frak m}[2]$ over $K^s$ correspond bijectively to the factors of even degree of $f(x,y)$ over $K^s$. Since the Galois group acts by permutation of the roots, we have a canonical isomorphism
$J_{\frak m}[2] \simeq (\Rlk\mu_2)_{N=1}.$
On the quotient $J$, there is a single relation: $(P_1) + (P_2) + \cdots + (P_{2g+2}) - (g+1)d =  {\rm div} (y) \equiv 0$, so
$J[2] \simeq (\Rlk\mu_2)_{N=1}/\mu_2.$ The last two statements of Proposition~\ref{prop:torsion} follow similarly.
\end{proof}

\medskip

Finally, we note that the Weil pairing $J[2] \times J[2] \rightarrow \mu_2$ gives the self-duality of the finite group scheme $(\Rlk\mu_2)_{N=1}/\mu_2$, and the connecting homomorphism
$H^1(K, J[2]) \rightarrow H^2(K,\mu_2)$ whose kernel is the image of $H^1(K,J_{\frak m}[2])$ is cup product with the class of $W[2]$ (see \cite[Proposition 10.3]{PSh}).

\section{Generic pencils of quadrics}\label{generic}

In this section, we relate hyperelliptic curves to pencils of quadrics. In particular, we will see how pencils of quadrics yield two-covers of $J^1$ for certain hyperelliptic curves.

Let $W=K^n$ be a vector space of dimension $n \geq 3$ over $K$ and let $A$ and $B$ be two symmetric bilinear forms on $W$.
Let $Q_A$ and $Q_B$ be the corresponding quadric hypersurfaces
in $\mathbb P(W)$, so $Q_A$ is defined by the equation $\langle w,w \rangle_A = 0$ and $Q_B$ is defined by the equation
$\langle w,w \rangle_B = 0$. Let $Y$ be the base locus of the pencil spanned by $A$ and $B$, which is defined by the equations
$\langle w,w \rangle_A  = \langle w,w \rangle_B = 0$ in $\mathbb P(W)$. Then $Y$ has dimension $n - 3$ and
is a smooth complete intersection if and only if the discriminant
of the pencil  $\disc(xA - yB) = f(x,y)$ has $\Delta(f) \neq 0$. In this case we say that the pencil spanned
by $A$ and $B$ is \emph{generic}. In this section, we will only consider generic pencils.
The Fano scheme $F = F(A,B)$ is the Hilbert scheme of maximal linear subspaces of $\mathbb P(W)$ that are contained in $Y$.

When $n = 2g+1$ is
odd, the Fano scheme has dimension zero and is a principal homogeneous space for the finite group scheme $\Rlk\mu_2/\mu_2 \simeq (\Rlk\mu_2)_{N=1}$. Here $L$ is the \'etale
algebra of rank~$2g+1$ determined by the binary form $f(x,y)$. The $2^{2g}$ points of $F$ over the separable closure of $K$
correspond to the subspaces
$Z$ of $W$ of dimension $g$ that are isotropic for all the quadrics in the pencil, and the scheme $F$ depends only on the $\SL_n(K)$-orbit of the
pair $(A,B)$.

When $n = 2g + 2$ is even, the Fano scheme $F$ is smooth and geometrically connected of dimension $g$, and is a principal homogeneous space for
the Jacobian $J$ of the smooth hyperelliptic curve $C$ with equation $z^2 = f(x,y)$. A point of $F$ corresponds to a subspace $Z$ of $W$ of dimension~$g$ that is isotropic for all of the quadrics in the pencil, whereas a point of $C$ corresponds to a quadric in the pencil plus a choice of one of the two rulings of that quadric. This interpretation can be used to define a morphism
$C \times F \rightarrow F$ over $K$, which in turn gives a simply transitive action of $J$ on $F$. In this case, the Fano variety $F$ depends only on the $(\SL_n/\mu_2)(K)$-orbit of the pair $(A,B)$. Proofs
of all assertions on the Fano scheme can be found in \cite {W}.

\begin{theorem}\label{thm:pencil}{\em \!(\cite[Theorem 2.7]{W})}
Let $F$ be the Fano variety of maximal linear subspaces contained in the base locus of a generic pencil of quadrics generated by symmetric bilinear forms $(A,B)\in K^2\otimes\Sym_2K^n.$ Let $f(x,y)$ denote the invariant binary form of $(A,B)$. Let $C:z^2=f(x,y)$ denote the corresponding hyperelliptic curve with Jacobian $J$. Then the disconnected variety
\begin{equation}\label{eq:pencilgeneric}
X := J \sqcup F \sqcup J^1\sqcup F
\end{equation}
has a commutative algebraic group structure over $K$. In particular, $[F]$ as a class in $H^1(K,J)$ is \emph{4}-torsion and $2[F]=[J^1].$
\end{theorem}

The group $X$ contains the subgroup $\Pic_{C/K}/\mathbb Z \cdot d = J \sqcup J^1$ with index two. Let $F[4]$ be the principal homogeneous space for $J[4]$ consisting of the points of $F$ of (minimal) order $4$ in the group~$X$. Multiplication by $2$ in $X$ gives finite \'etale covers
$$F \rightarrow J^1$$
$$F[4] \rightarrow W[2]$$
of degree $2^{2g}$ with an action of the group scheme $J[2]$. This shows that the class $[F]$ of the principal homogeneous space $F$ satisfies $2[F] = [J^1]$ in the group $H^1(K,J)$. Similarly, the class of $W[2]$ in $H^1(K,J[2])$ is the image of the class $F[4]$ in $H^1(K,J[4])$ under the map $m_2: H^1(K,J[4]) \rightarrow H^1(K,J[2])$ induced by the multiplication by $2$ map from $J[4]$ to $J[2]$. In general, if an element $[F']\in H^1(K,J[2])$ is in the image of $m_2$, we say $[F']$ is divisible by $2$ in $H^1(K,J[4])$.

Consequently, a necessary condition on the existence of a pencil $(A,B)$  over $K$ with discriminant curve $C$ is that the class of $J^1$ and the class of $W[2]$ should be divisible by $2$ in $H^1(K,J)$ and $H^1(K,J[4])$ respectively. However, this condition is not sufficient. Consider the curve $C$ of genus zero with equation $z^2 = -x^2 - y^2$ over $\mathbb R$. In this case, both $J$ and $J[2]$ reduce to a single point, so any homogeneous space for $J$ or $J[2]$ is trivial, and hence divisible by $2$. On the other hand, since $L = \mathbb C$ and $f_0 = -1$ is not a norm, by Corollary~\ref{cor:orbitSL} (or \ref{cor:orbitSLmu2}) there are no pencils over $\mathbb R$ with discriminant $f(x,y) = -x^2 - y^2$. To obtain a geometric condition that is both necessary and sufficient for the existence of a pencil, we will have to consider non-generic pencils whose invariant binary form defines the singular curve $C_{\frak m}$. This is the object of the next section.

\section{Regular pencils of quadrics}\label{regular}

In this section, we give a list of equivalent conditions for the existence of a pencil over $K$ whose discriminant is some given binary form $f(x,y)$. In particular, we prove Theorem \ref{genJ}.

Let $(A,B)$ generate a generic pencil of bilinear forms on a vector space $W$ of even dimension $n = 2g+2$ over $K$, and let $f(x,y) = \disc(xA - yB)$ be its invariant binary form of degree $2g+2$ and discriminant $\Delta(f) \neq 0$. We continue to assume that $f_0 = \disc(A)$ is also nonzero in $K$. Let $(A',B')$ be a pair of bilinear forms on the vector space $W' = W \oplus K^2$ of dimension $n + 2 = 2g + 4$, where $A'$ is the direct sum of $A$ and the rank one form $\langle(a,b),(a',b')\rangle = aa'$ on $K^2$ and $B'$ is the direct sum of $B$ and the split form $\langle(a,b),(a',b')\rangle = ab' + a'b$ of rank $2$. The invariant binary form of this pencil
$$\disc(xA' - yB') = f(x,y)y^2$$
then has a double zero at $(x,y) = (1,0)$, and the pencil is not generic. The base locus defined by the equations $Q_{A'} = Q_{B'} = 0$ in $\mathbb P(W \oplus K^2)$ has an ordinary double point at the unique singular point $R = (0_W; 0,1)$ of the quadric $Q_{A'}$. There are exactly $2g+3$ singular quadrics in the pencil and all of them are simple cones. The $K$-algebra $L'$ associated to the pencil is not \'etale, but is isomorphic to $L \oplus K[y]/y^2$. Even though $L'$ is not \'etale, the vector space $W'$ is a free $L'$-module of rank 1, so the pencil is regular in the sense of \cite[\S3]{W}. Since the norms from $K[y]/y^2$ to $K$ are precisely the squares in $K$, we have an equality of quotient groups $K^\times/(K^{\times2}N(L^\TIMES)) = K^\times/(K^{\times2}N(L'^\times))$.

The Fano scheme $F_\frak {m}$ of this pencil consists of the subspaces $Z$ of dimension $g+1$ in $W \oplus K^2$ that are isotropic for all of the quadrics in the pencil and do not contain the unique line that is the radical of the form $A'$ (so the projective space $\mathbb P(Z)$, which is contained in the base locus, does not meet the unique double point $R$). The Fano scheme is a smooth variety of dimension $g+1$. However, in this case $F_{\frak m}$ is not projective. It is a principal homogeneous space for the generalized Jacobian $J_{\frak m}$ associated to the singular curve $C_{\frak m}$ of arithmetic genus $g +1$ and equation $z^2 = f(x,y)y^2$ in weighted projective space.

For example, when $g = 0$, the curve $C$ is the non-singular quadric $z^2 = ax^2 + bxy + cy^2$ in $\mathbb P^2$, with $a = f_0$ and $b^2 - 4ac = \Delta(f)$ both nonzero in $K$.
The pencil $(A',B')$ has discriminant $f'(x,y) = ax^2y^2 + bxy^3 + cy^4$. Its base locus $D$ in $\mathbb P^3$ is isomorphic to a singular curve of arithmetic genus one, with a single node $R$ whose tangents are rational over the quadratic extension $K' = K(\sqrt{f_0})$. The Fano variety $F_{\frak m}$ in this case is just the affine curve $D - \{R\}$, and $J_{\frak m}^1$ is the affine curve $C_{\frak m} - \{Q\} = C - \{P,P'\}$. Both are principal homogeneous spaces for the one-dimensional torus $T = J_{\frak m}$ which is split by $K'$. We shall see that there is an unramified double cover $F_{\frak m} \rightarrow J_{\frak m}^1$ that extends to a double cover of complete curves of genus zero $M \rightarrow C$ which is ramified at $P$ and $P'$.

Since the pencil is regular and its associated hyperelliptic curve has only nodal singularities, we again obtain a commutative algebraic group
\begin{equation}\label{eq:pencilregular}
X_{\frak m} = J_{\frak m} \sqcup F_{\frak m} \sqcup J_{\frak m}^1\sqcup F_{\frak m}
\end{equation}
over $K$ with connected component $J_{\frak m}$ and component group $\mathbb Z/4$. The group $X_{\frak m}$ contains the algebraic group $\Pic_{C_{\frak m}/K}/\mathbb Z \cdot d = J_{\frak m} \sqcup J_{\frak m}^1$ with index two \cite[ \S 3.2]{W}. Just as in the generic case, multiplication by $2$ in the group $X_{\frak m}$ gives an unramified cover
$$F_{\frak m} \rightarrow J_{\frak m}^1$$
of degree $2^{2g+1}$ with an action of $J_{\frak m}[2]$, and shows that $2[F_{\frak m}] = [J_{\frak m}^1]$ in the group $H^1(K,J_{\frak m})$ of principal homogeneous spaces for $J_{\frak m}$. Hence a necessary condition for the existence of such a pencil $(A',B')$ is that the class of $J_{\frak m}^1$ is divisible by $2$. In this case, the necessary condition is also sufficient.

\begin{theorem}\label{thm:with8equiv}
Let $f(x,y) = f_0x^{2g+2} + f_1x^{2g+1}y + \cdots + f_{2g+2}y^{2g+2}$ be a binary form of degree $2g +2$ over $K$ with $\Delta (f)$ and $f_0$ both nonzero. Write $f(x,1) = f_0g(x)$ with $g(x)$ monic and separable. Let $L$ be the \'etale algebra $K[x]/ g(x)$ of degree n over $K$ and let $\beta$ denote the image of $x$ in $L$. Let $C$ be the smooth hyperelliptic curve of genus $g$ with equation $z^2 = f(x,y)$ and let $C_{\frak m}$ be the singular hyperelliptic curve of arithmetic genus $g+1$ with equation $z^2 = f(x,y)y^2$. Then the following conditions are all equivalent:\
\medskip

\noindent a. There is a generic pencil $(A,B)$ over $K$ with $\disc(xA - yB) = f(x,y)$.\
\medskip

\noindent b. There is a regular pencil $(A',B')$ over $K$ with $\disc(xA' -yB') = f(x,y)y^2$.\
\medskip

\noindent c. The coefficient $f_0$ lies in the subgroup $K^{\times2}N(L^\TIMES)$ of $K^\times$.\
\medskip

\noindent d. The class of the homogeneous space $J_{\frak m}^1$ is divisible by $2$ in the group $H^1(K,J_{\frak m})$.\
\medskip

\noindent e. The class of the homogeneous space $W_{\frak m}[2]$ is in the image of the map $m_2':H^1(K,J_{\frak m}[4]) \rightarrow H^1(K,J_{\frak m}[2])$ induced by the multiplication by $2$ map from $J_{\frak m}[4]$ to $J_{\frak m}[2]$.\
\medskip

\noindent f. There is an unramified two-cover of homogeneous spaces $F_{\frak m} \rightarrow J_{\frak m}^1$~~ for $J_{\frak m}$ over $K$.\
\medskip

\noindent g. The maximal unramified abelian cover $U \rightarrow C - \{P,P'\}$ of exponent $2$ over $K^s$ descends to $K$.\
\medskip

\noindent h. The maximal abelian cover $M \rightarrow C$ of exponent $2$ over $K^s$ that is ramified
only at the points $\{P,P'\}$ descends to $K$.\
\end{theorem}

Note that the maximal abelian covers above all have degree $2^{2g+1}$. The equivalence of conditions $a,d$, and $f$ proves Theorem~\ref{genJ}.

\medskip

\begin{proof}
$c\Leftrightarrow a \Rightarrow b$. We have already seen the equivalence of $a$ and $c$ in Corollary~\ref{cor:orbitSL}. The implication $a\Rightarrow b$ is obvious from the construction of the regular pencil $(A',B')$ from a generic pencil $(A,B)$ earlier in this section.

$b\Rightarrow d \Leftrightarrow e \Leftrightarrow f.$ When a regular pencil $(A',B')$ over $K$ with $\disc(xA' -yB') = f(x,y)y^2$ exists, the Fano variety $F_{\frak m}$ of the base locus of this pencil
provides a homogenous space for $J_{\frak m}$ whose class is a square root of the class of $J_{\frak m}^1$ in the group $H^1(K,J_{\frak m})$.
The equivalence of conditions $d,e$ and $f$ is clear.

$f\Rightarrow g\Rightarrow h.$ Assuming that a two-cover $F \rightarrow J_{\frak m}^1$ exists over $K$, we obtain the maximal unramified abelian cover of $C - \{P,P'\}$ by taking the fiber product with the morphism $C - \{P,P'\} \rightarrow J_{\frak m}^1$, and the maximal abelian cover of $C$ ramified only at the points $\{P,P'\}$ by taking the closure of the above unramified cover of $C - \{P,P'\}$.

$h\Rightarrow c.$ Finally, assuming the existence of the maximal abelian cover $M\rightarrow C$ of exponent $2$ that is ramified only at the points $\{P,P'\}$, we show that $f_0$ lies in the subgroup $K^{\times2}N(L^\TIMES)$ of $K^\times$,
which will complete the proof of Theorem \ref{thm:with8equiv}. The cover $M\rightarrow C$ corresponds to an inclusion of function fields $K(C) \rightarrow K(M)$. Over $K^s$, the function field $K^s(M)$ is obtained from $K^s(C)$ by adjoining the square roots of all rational functions on $C$ whose divisors have the form $2d_1$ or $2d_1 + (P) + (P')$ for some divisor $d_1$ on $C$. Since the characteristic of $K$ is not equal to $2$, these square roots give either unramified covers of $C$ or covers that are ramified only at the two points $P$ and $P'$ where the ramification is tame. More precisely, there are
$2^{2g+1} - 1$ distinct quadratic extensions of $K^s(C)$ of this form that are contained in $K^s(M)$, and their composition is equal to $K^s(M)$.

Indeed, by Galois theory, these quadratic extensions correspond to the subgroups of index $2$ in $J_{\frak m}[2](K^s)$, or equivalently to nontrivial $K^s$-points in the Cartier dual $\Rlk\mu_2/\mu_2$. Let $w$ be the rational
function $z/y^{g+1}$ on $C$, and let $t$ be the rational function $x/y$ on $C$, so $w^2 = f_0g(t)$. The nontrivial points in $(\Rlk\mu_2/\mu_2)(K^s)$ correspond bijectively to the nontrivial monic factorizations $g(x) = h(x)j(x)$ over $K^s$, and the corresponding quadratic extension of $K^s(C)$ is given by $K^s(C)(\sqrt {h(t)}) = K^s(C)(\sqrt {j(t)})$. When both $h(x)$ and $j(x)$ have even degree, the divisors of the rational functions $h(t)$ and $j(t)$ are of the form $2d_1$ and the corresponding quadratic cover of the curve $C$ is unramified. When the factors both have odd degree, these divisors are of the form $2d_1 + (P) + (P')$ and the quadratic cover is ramified at the points $P$ and $P'$.

Since there might be no nontrivial factorizations of $g(x)$ over $K$, there might be no nontrivial $K$-rational points of $\Rlk\mu_2/\mu_2$ and hence no quadratic field extensions of $K(C)$ contained in $K(M)$. However, over
$L$ we have the factorization $g(x) = (x - \beta)j(x) = h(x)j(x)$, so the algebra $L(M)$ must contain a square root $u$ of some constant multiple of the function $h(t) = (t - \beta)$. (The need to adjoin a square root of $t-\beta$ whose divisor has the form $2d_1+(P)+(P')$ is the main reason for the appearance of the generalized Jacobian $J_{\frak m}$ (cf. \cite[Footnote 2]{PSh}).)
Write $u^2 = \alpha(t - \beta)$ with $\alpha$ in $L^\TIMES$ and take the norm to $K(M)$ to obtain the equation $N(u)^2 = N(\alpha)g(t)$. Then the two rational functions $N(u)$ and $w$ in $K(M)^\times$ have the same
divisor, so they are equal up to a constant factor in $K^\times$. Writing $bN(u) = w$ with $b$ in $K^\times,$ we find $w^2 =  b^2 N(u)^2 = b^2N(\alpha)g(t)$. However, $w^2 = f_0g(t)$, so $f_0 = b^2N(\alpha)$ is in the subgroup $K^{\times2}N(L^\TIMES)$ of $K^\times$. This completes the proof of Theorem \ref{thm:with8equiv}.
\end{proof}

In fact, the obstruction classes for the eight conditions in Theorem \ref{thm:with8equiv} are all equal. More precisely, the obstruction class for conditions $a,b,c$ is the class of $f_0$ in $K^\times/(K^{\times2}N(L^\TIMES))$. This group can be viewed as a subgroup of $H^2(K,J_{\frak m}[2])$ via
$$\text{coker}\!\left(N:H^1(K,\Rlk\mu_2)\rightarrow H^1(K,\mu_2)\right)\,\longhookrightarrow \,H^2(K, (\Rlk\mu_2)_{N=1}).$$
We denote the image of $f_0$ in $H^2(K,J_{\frak m}[2])$ by $[f_0].$ This is the cohomological class $d_f$ whose non-vanishing obstructs the existence of rational orbits with invariant binary form $f$ for (all pure inner forms of) $\SL_n$; see \cite[\S 2.4 and Theorem~9]{AITII}).

The obstruction class for conditions $d,e$ is the class $\delta[J^1_{\frak m}]$ in $H^2(K,J_{\frak m}[2])$ where $\delta$ is the connecting homomorphism
$H^1(K,J_{\frak m})\rightarrow H^2(K,J_{\frak m}[2])$ arising from the exact sequence $1\rightarrow J_{\frak m}[2]\rightarrow J_{\frak m}\xrightarrow{2} J_{\frak m}\rightarrow 1.$

The obstruction class for conditions $f,g,h$ comes from Galois descent. There is an unramified two-cover $\pi:J_{\frak m}^1\rightarrow J_{\frak m}^1$ over $K^s$ obtained by identifying $J_{\frak m}^1$ with $J_{\frak m}$ using a $K^s$-point of $J_{\frak m}^1$, then taking the multiplication-by-2 map on $J_{\frak m}.$ The descent obstruction of this cover to $K$ is the image in $H^2(K,J_{\frak m}[2])$ of the class $[\pi:J_{\frak m}^1\rightarrow J_{\frak m}^1]$ under the following map from the Hochschild-Serre spectral sequence:
$$H^0\left(K,H^1(C\times_KK^s - \{P,P'\},J_{\frak m}[2])\right)\longrightarrow H^2(K,J_{\frak m}[2]).$$
This obstruction class equals $\delta[J^1_{\frak m}]$ for formal reasons (cf.\ \cite[Lemma 2.4.5]{Sk}). We have the following strengthening of Theorem \ref{thm:with8equiv}.

\begin{theorem}\label{thm:equalityobstruction}
Let $f(x,y) = f_0x^{2g+2} + f_1x^{2g+1}y + \cdots + f_{2g+2}y^{2g+2}$ be a binary form of degree $2g +2$ over $K$ with $\Delta (f)$ and $f_0$ both nonzero. Let $C$ be the smooth hyperelliptic curve of genus $g$ with equation $z^2 = f(x,y)$ and let $J_{\frak m}$ denote its generalized Jacobian. Then the obstruction classes for conditions $a$ through $h$ in Theorem~$\ref{thm:with8equiv}$ are all equal in $H^2(K,J_{\frak m}[2])$, i.e., $[f_0]=\delta[J^1_{\frak m}]$.
\end{theorem}

\begin{proof} Consider the following commutative diagram:
\begin{displaymath}
\xymatrix{
1\ar[r]&J_{\frak m}[2]\ar[r]\ar@{^{(}->}[d]&J_{\frak m}\ar[r]^{2}\ar@{^{(}->}[d]&J_{\frak m}\ar[r]\ar[d]^{=}&1\\
1\ar[r]&(J_{\frak m} \sqcup J_{\frak m}^1)[2]\ar[r]\ar@{->>}[d]&J_{\frak m} \sqcup J_{\frak m}^1\ar[r]^{2}\ar@{->>}[d]&J_{\frak m}\ar[r]&1\\
&\mu_2\ar[r]^{=}&\mu_2&&.}
\end{displaymath}
Here the map $J_{\frak m} \sqcup J_{\frak m}^1\xrightarrow{2} J_{\frak m}$ is given by $[D]\mapsto 2[D]-\deg([D])\cdot d.$ Theorem \ref{thm:equalityobstruction} follows from the following two results.

\begin{proposition}\label{prop:W2div}
For any $a\in K,$ there exists a class $[J^{1/2}_a]\in H^1(K,J_{\frak m} \sqcup J_{\frak m}^1)$ such that $2[J^{1/2}_a]=[J^1_{\frak m}]$ in $H^1(K,J_{\frak m})$ and such that the image of $[J^{1/2}_a]$ in $H^1(K,\mu_2)=K^\times/K^{\times2}$ equals $f_0g(a)=f_0N_{L/K}(a-~\beta)$.
\end{proposition}

\begin{lemma}\label{lem:commute}
Let $1\rightarrow A_1\rightarrow B_1\rightarrow C\rightarrow 1$ and $1\rightarrow A_2\rightarrow B_2\rightarrow C\rightarrow 1$ be central extensions of algebraic groups over $K$ such that the following diagram commutes:
\begin{displaymath}
\xymatrix{
1\ar[r]&A_1\ar[r]\ar@{^{(}->}[d]&B_1\ar[r]\ar@{^{(}->}[d]&C\ar[r]\ar[d]^{=}&1\\
1\ar[r]&A_2\ar[r]\ar@{->>}[d]&B_2\ar[r]\ar@{->>}[d]&C\ar[r]&1\\
&D\ar[r]^{=}&D&&}
\end{displaymath}
Then the following diagram commutes up to sign:
\begin{displaymath}
\xymatrix{
H^1(K,B_2)\ar[r]\ar[d]&H^1(K,C)\ar[d]\\
H^1(K,D)\ar[r]&H^2(K,A_1)}
\end{displaymath}
\end{lemma}

Lemma \ref{lem:commute} follows from a direct cocycle computation. For more details, see \cite[Lemma~2.8.2]{Thang}. We note that when Lemma \ref{lem:commute} is used to prove Theorem \ref{thm:equalityobstruction}, all the cohomology groups are 2-torsion and hence commutativity up to sign is equivalent to commutativity. We now prove Proposition \ref{prop:W2div}. Fix $a\in K.$ Let $P_a\in C(K(\sqrt{\al}))$ be a point with $x$-coordinate $a$, where $\al=f_0g(a)$, and let $P'_a$ be the conjugate of $P_a$ under the hyperelliptic involution. The class $[J^1_{\frak m}]\in H^1(K,J_{\frak m})$ is given by the 1-cocycle $\sigma\mapsto \gl{(P'_a)}{\sigma}-(P'_a).$ In other words,
$$[J^1_{\frak m}]_\sigma = \begin{cases}0&\text{if }\sigma(\sqrt{\al})=\sqrt{\al}\\(P_a)-(P'_a)&\text{if }\sigma(\sqrt{\al})=-\sqrt{\al}.\end{cases}$$
Let $[J^{1/2}_a]$ denote the following 1-cochain with values in $(J_{\frak m} \sqcup J^1_{\frak m})(K^s):$
$$[J^{1/2}_a]_\sigma = \begin{cases}0&\text{if }\sigma(\sqrt{\al})=\sqrt{\al}\\(P_a)&\text{if }\sigma(\sqrt{\al})=-\sqrt{\al}.\end{cases}$$
Since $2(P_a)- d =2(P_a)-((P_a)+(P'_a))=(P_a)-(P'_a),$ we see that $2[J^{1/2}_a]_\sigma=[J^1_{\frak m}]_\sigma$ for all $\sigma\in\Gal(K^s/K).$ Moreover, a direct computation shows that $[J^{1/2}_a]$ is a 1-cocycle and its image in $H^1(K,\mu_2)$ is the 1-cocycle $\sigma\mapsto\gl{\sqrt{\al}}{\sigma}/\sqrt{\al}.$ This completes the proof of Proposition \ref{prop:W2div}, and thus Theorem \ref{thm:equalityobstruction}.
\end{proof}

\section{Soluble orbits} \label{soluble}

In the previous section, we gave necessary and sufficient conditions for the existence of pencils of bilinear forms $(A,B)\in K^2\otimes\Sym_2 K^n$ having a given invariant binary form. In this section, we consider \textit{soluble} pencils of bilinear forms $(A,B),$ i.e., those for which the associated Fano variety $F=F(A,B)$ has a $K$-rational point.

Fix a binary form $f(x,y)$ of degree $n = 2g+2$ over $K$ with $\Delta(f)$ and $f_0$ nonzero in $K$, and let $C$ be the smooth hyperelliptic curve with equation $z^2 = f(x,y)$. Suppose that $(A,B)$ is a generic pencil of bilinear forms on $W$ over $K$ with invariant binary form $f(x,y)=\disc(Ax-By)$ and let $(A',B')$ be the regular pencil
of bilinear forms on $W \oplus K^2$ having invariant binary form $f(x,y)y^2$ constructed in Section \ref{regular}.  We say that $(A,B)$ lies in
a \textit{soluble} orbit for $\SL_n$ if the Fano variety $F_{\frak m}$ of the base locus of $(A',B')$ has a $K$-rational point. Similarly, we say that the pencil $(A,B)$ lies in a {\it soluble} orbit for $\SL_n/\mu_2$ if the Fano variety $F$ of the base locus of $(A,B)$ has a $K$-rational point. In this section, we classify the soluble orbits for $\SL_n$ and $\SL_n/\mu_2$.

Since we have constructed an unramified two-cover
$F_{\frak m} \rightarrow J_{\frak m}^1$,
a necessary condition for the existence of soluble orbits for $\SL_n$ is that $J_{\frak m}^1(K)$ is nonempty. 
In this case, the group $J_{\frak m}(K)$ acts simply transitively on the set of points $J_{\frak m}^1(K)$.

\begin{theorem}\label{thm:sol}
Let $f(x,y)$ be a binary form of degree $n=2g+2$ over $K$ with $\Delta(f)$ and $f_0$ nonzero in $K$. Then soluble orbits for the action of $\SL_n(K)$ on $K^2\otimes\Sym_2K^n$ having invariant binary form $f(x,y)$ exist if and only if there is a $K$-rational divisor of odd degree on the curve $C:z^2 = f(x,y)$. In that case, they are in bijection with the elements of $J^1_{\frak m}(K)/2J_{\frak m}(K)$.
\end{theorem}

\begin{proof}
Suppose first that soluble orbits with invariant binary form $f(x,y)$ exist. Let $(A,B)$ be in $K^2\otimes\Sym_2K^n$ with invariant binary form $f(x,y)$ such that the Fano variety $F(A,B)_{\frak m}$ of the associated regular pencil $(A',B')$ in $W\oplus K^2$ has a rational point. The stabilizer of $(A,B)$ in $\SL_n$ is isomorphic to $J_{\frak m}[2]$ by Corollary~\ref{cor:orbitSL} and Proposition~\ref{prop:torsion}. Since $H^1(K,\SL_n)=1,$ we see that the rational orbits with invariant binary form $f(x,y)$ are in bijection with the elements in the Galois cohomology group $H^1(K,J_{\frak m}[2])$. This bijection depends on the choice of the initial soluble orbit $(A,B)$ which maps to the trivial class in  $H^1(K,J_{\frak m}[2])$.

Explicitly, suppose the pair $(A_1,B_1)\in K^2\otimes\Sym_2K^n$ has invariant binary form $f(x,y)$ and corresponds to the class $c\in H^1(K,J_{\frak m}[2])$. Let $(A_1',B_1')$ be the associated regular pencil with Fano variety $F(A_1,B_1)_{\frak m}.$ Then as elements of $H^1(K,J_{\frak m})[4],$ we have, up to sign\footnote{The ambiguity of sign comes from the fact that we cannot distinguish between $[F_{\frak m}]$ and $-[F_{\frak m}]$ in $H^1(K,J_{\frak m})$. In other words, we cannot distinguish the two copies of $F_{\frak m}$ in the group $X_{\frak m}$ defined in \eqref{eq:pencilregular}.\label{ft:sign}}, the formula
\begin{equation}\label{eq:soleq}
[F(A_1,B_1)_{\frak m}] = [F(A,B)_{\frak m}] + j'(c),
\end{equation}
where $j'$ denotes the natural map $H^1(K,J_{\frak m}[2])\rightarrow H^1(K,J_{\frak m})[2]$ and the addition is taking place in $H^1(K,J_{\frak m}).$
Hence we see that $F(A_1,B_1)_{\frak m}$ is the trivial torsor of $J_{\frak m}$ if and only if $c$ is in the Kummer image of $J_{\frak m}(K)/2J_{\frak m}(K).$ Therefore, the set of soluble orbits with invariant binary form $f(x,y)$ forms a principal homogeneous space for the quotient group $J_{\frak m}(K)/2J_{\frak m}(K).$ The choice of the fixed soluble orbit $(A,B)$ trivializes this principal homogeneous space.

On the other hand, if $x\in F(A,B)_{\frak m}(K)$ is any rational point, then the sum $x+x = 2x$ in the algebraic group $X_{\frak m}$ in \eqref{eq:pencilregular} gives a rational point of $J^1_{\frak m}$ well-defined up to hyperelliptic conjugation (cf. Footnote \ref{ft:sign}). Hence $J^1_{\frak m}(K)$ is nonempty. Therefore, the set $J^1_{\frak m}(K)/2J_{\frak m}(K)$ is also in bijection with $J_{\frak m}(K)/2J_{\frak m}(K)$.

To complete the proof of Theorem \ref{thm:sol}, it remains to show that if $J^1_{\frak m}(K)$ is nonempty, then soluble orbits with invariant binary form $f(x,y)$ exist.
We show this first in the special case where the curve $C_{\frak m}$ has a non-singular $K$-rational point $Q=(x_0,1,z_0)$. Let $L=K[x]/f(x,1)$ denote as usual the \'{e}tale algebra of rank $n$ associated to $f(x,y)$ and let $\beta$ denote the image of $x$ in $L$. The rational orbit corresponding to $(Q)$ is given by the equivalence class of a pair $(\alpha,s)$ (see Corollary~\ref{cor:orbitSL}) where $\alpha = (x-T)(Q)$. Here ``$x-T$'' is the descent map introduced by Cassels \cite{Cassels}:
$$J^1_{\frak m}(K)/2J_{\frak m}(K) \rightarrow (L^\times/L^{\times2})_{N \equiv f_0}.$$
We note that $s$ is not uniquely determined when $W_{\frak m}[2]$ is a nontrivial torsor of $J_{\frak m}[2].$ In this case, the fibers of the above $x-T$ map also have size $2$. From the definition of the bijection between the set of rational orbits and the set of equivalence classes of pairs $(\alpha,s)$ in Section \ref{dedorbits}, we see that if the orbit corresponding to a pair $(\alpha,s)$ is soluble, then the orbit corresponding to any pair $(\alpha',s')$ with $\alpha'=\alpha$ is also soluble.

Consider the two bilinear forms $(A',B')$ on $L\oplus K^2$ given by
\begin{eqnarray*}
\langle (\lambda,a,b),(\mu,a',b') \rangle_{A'} &=& (\mbox{coefficient of }\beta^{n-1}\mbox{ in }\alpha\lambda\mu) + aa',\\
\langle (\lambda,a,b),(\mu,a',b') \rangle_{B'} &=& (\mbox{coefficient of }\beta^{n-1}\mbox{ in }\alpha\beta\lambda\mu) + ab' + a'b.
\end{eqnarray*}
We show that for $\alpha=(x-T)(Q)$, there is a rational $(g+1)$-plane $x'$ isotropic with respect to both bilinar forms.

When $z_0\neq0$, we have $\alpha=(x-T)(Q)=x_0-\beta.$ Then $$x'=\Span\{(1,0,0),(\beta,0,0),\ldots,(\beta^{g-1},0,0),(\beta^g,1,-\frac{1}{2}(x_0+\frac{f_1}{f_0}))\}$$ is isotropic with respect to both bilinear forms. To check this, we note that the unique polynomial $P(x)$ of degree at most $2g+1$ with $P(\beta)=(x_0-\beta)\beta^{2g+1}$ has leading coefficient $x_0+f_1/f_0.$

When $z_0=0$, we set $h_0(t)=t-x_0$ and $h_1(t)=f(t,1)/(t-x_0)$. Then $\alpha=(x-T)(Q)=h_1(\beta)-h_0(\beta)$ and the following $(g+1)$-plane is isotropic with respect to both bilinear forms: $$x'=\Span\{(h_1(\beta)-h_1(x_0),0,0),(\beta-x_0,0,0),\ldots,((\beta-x_0)^{g-1},0,0),((\beta-x_0)^g,1,-\frac{1}{2}((2g+1)x_0 + \frac{f_1}{f_0}))\}.$$ This can be checked by a simple calculation noting that $h_1(\beta)(h_1(\beta)-h_1(x_0))=h_0(\beta)h_1(\beta)=0.$

Before moving on to the general case, we make an important observation. Using this pencil with $\alpha=(x-T)(Q)$ as the base point, we obtain a bijection between the set of the rational orbits with invariant binary form $f(x,y)$ and $H^1(K,J_{\frak m}[2])$ as described above. If $(A_1,B_1)$ is an element of $K^2\otimes\Sym_2K^n$ with invariant binary form $f(x,y)$ such that its associated $\alpha$ equals $(x-T)(D)$ for some $D\in J^1_{\frak m}(K)/2J_{\frak m}(K)$, then the orbit of $(A_1,B_1)$ corresponds to the class $D - (Q)$ or $D - (Q')$ in $J_{\frak m}(K)/2J_{\frak m}(K)$ where $Q'$ denotes the hyperelliptic conjugate of $Q$. Hence the orbit of $(A_1,B_1)$ is soluble.

We now treat the general case, assuming only that $J^1_{\frak m}(K)$ is nonempty. Now $C_{\frak m}$ has a non-singular point $Q$ defined over some extension $K'$ of $K$ of odd degree $k$. Let $D\in J^1_{\frak m}(K)$ denote the divisor class of degree $1$ obtained by taking the sum of the conjugates of $Q$ and subtracting $\frac{k-1}{2}$ times the hyperelliptic class. We claim that the orbits corresponding to $D$ are soluble thereby completing the proof of Theorem \ref{thm:sol}. To prove the claim, let $(A,B)$ be an element of $K^2\otimes\Sym_2K^n$ with invariant binary form $f(x,y)$ such that its associated $\alpha$ equals $(x-T)(D)$, and let $F(A,B)_{\frak m}$ denote the Fano variety of the associated regular pencil. Since $C$ has a point over $K'$, we have seen that the $K'$-rational orbits $(\alpha,s)$ with $\alpha=(x-T)(Q)$ and hence with $\alpha=(x-T)(D)$ are soluble over $K'$. In other words, $F(A,B)_{\frak m}(K')$ is nonempty. Thus, as an element of $H^1(K,J_{\frak m})$, the class of $F(A,B)_{\frak m}$ becomes trivial when restricted to $H^1(K',J_{\frak m})$. A standard argument using the corestriction map shows that this class is killed by the degree $k$ of $K'$ over $K$. Since $F(A,B)_{\frak m}$ is a torsor of $J_{\frak m}$ of order dividing $4$ and $k$ is odd, we see that $F(A,B)_{\frak m}$ must be the trivial torsor.
\end{proof}

\medskip

The same argument also classifies the soluble orbits for $\SL_n/\mu_2$, provided that $C$ has a $K$-rational divisor of odd degree. The descent map ``$x-T$'' gives a map of sets
$$J^1(K)/2J(K) \rightarrow (L^\TIMES/(L^{\TIMES2}K^\times))_{N \equiv f_0}$$
and is either 2-to-1 or injective (depending on the triviality of the class $W[2]$ in $H^1(K,J[2])$). To see that there are no soluble orbits when $C$ has no divisors of odd degree, we use the exact sequence of commutative algebraic groups \cite[Corollary~3.22]{W}:
$$1 \rightarrow T \rightarrow X_{\frak m} \rightarrow X \rightarrow 1.$$
If $J_{\frak m}^1(K)$ is empty but both $J^1(K)$ and $F(K)$ are nonempty, then the quotient of $X(K)$ by the image of $X_{\frak m}(K)$ maps onto the component group $\mathbb Z/ 4\mathbb Z$ of
$X$. On the other hand, this quotient injects into $H^1(K,T)$, which has exponent $2$, a contradiction. Hence we have proved the following:

\begin{theorem}\label{thm:soluble}
Let $f(x,y)$ be a binary form of degree $n=2g+2$ over $K$ with $\Delta(f)$ and $f_0$ nonzero in $K$. Then soluble orbits for the action of $(\SL_n/\mu_2)(K)$ having invariant binary form $f(x,y)$ exist if and only if there is a $K$-rational divisor of odd degree on the curve $C:z^2=f(x,y)$. In that case, they are in bijection with the cosets of $J^1(K)/2J(K)$ and the group $J(K)/2J(K)$ acts simply transitively on the set of soluble orbits.
\end{theorem}

\section{Finite fields and archimedean local fields}\label{arithmetic}

In this section we consider the orbits for the action of $(\SL_n/\mu_2)(K)$ on $K^2\otimes\Sym_2K^n$ when the base field $K$ is a finite field or an archimedean local field. In particular, we compute the number of these orbits with a fixed invariant binary form $f(x,y)$.

\subsection{Finite fields}\label{sec:Fporbits}

Let $K$ be a finite field of odd cardinality $q$. Let $f(x,y)$ be a binary form of even degree $n$ over $K$
with nonzero discriminant $\Delta$ and nonzero first coefficient $f_0$, and write $f(x,1) = f_0g(x)$. We factor
$$g(x) = \prod_{i=1}^m g_i(x)$$
where $g_i(x)$ has degree $d_i$ and is irreducible. Then $L$ is the product of $m$ finite fields $L_i$ of cardinality $q^{d_i}$. Since finite fields have unique extensions of any degree, we see that either one of the $L_i$ has odd degree over $K$ or all of the $L_i$ contain the unique quadratic extension of $K$. Therefore, $f(x,y)$ always has either an odd or an even factorization over $K$.

Since the norm map $L^\TIMES \rightarrow K^\times$ is surjective, $f_0$ is always a norm. By Corollary \ref{cor:orbitSLmu2}, the number of $(\SL_n/\mu_2)(K)$-orbits
with binary form $f(x,y)$ 
is: $2^m$ if all $L_i$ have even degree and $n\equiv0$~(mod $4$); $2^{m-1}$ if all $L_i$ have even degree and $n\equiv2$~(mod $4$); and $2^{m-2}$ if some $L_i$ has odd degree over $K$. The size of the stabilizer equals the number of even factorizations of $f(x,y)$ over $K$. Hence the stabilizer has size given as follows: $2^m$ if all $L_i$ have even degree and $n\equiv0$ (mod~$4$); $2^{m-1}$ if all $L_i$ have even degree and $n\equiv2$ (mod $4$); and $2^{m-2}$ if some $L_i$ has odd degree over $K$. Therefore, the number of pairs $(A,B)\in K^2\otimes\Sym_2K^n$ with invariant binary form $f(x,y)$ is $|(\SL_n/\mu_2)(K)|=|\!\SL_n(K)|$. This agrees with \cite[\S 3.3]{B}. For the purpose of application in Section~\ref{mainproofs}, the main ingredients that we need are the number of orbits and the fact that all the orbits with the same invariant binary form have the same number of elements.

By Lang's theorem, we have $H^1(K,J) = H^1(K,J_{\frak m}) = 0$. Hence the Fano
varieties $F$ and $F_{\frak m}$ associated to an orbit always have a $K$-rational point, and
every orbit is soluble.

\subsection{$\R$ and $\C$}\label{sec:RandCorbits}

We now classify the orbits over $K = \R$ and $K = \C$. Let $f(x,y)$ be a binary form of degree $n$ over~$K$
with nonzero discriminant $\Delta$ and nonzero first coefficient $f_0$, and write $f(x,1) = f_0g(x)$. Over $\C$ there is a single orbit with binary form $f(x,y)$.

In the case when $K = \R$, we factor
$$g(x) = \prod_{i=1}^{r_1} g_i(x) \prod_{j =1}^{r_2} h_j(x)$$
where each $g_i(x)$ has degree one while each $h_j(x)$ has degree two and is irreducible. Then the
algebra~$L$ is the product of
$r_1$ copies of $\R$ and $r_2$ copies of $\C$, with $r_1 + 2r_2 = n$. Note that $r_1$ has the
same parity as $n$, so is always even. The quotient group $\R^\times/(\R^{\times2}N(L^\TIMES))$ is trivial unless $r_1 = 0$, in which
case it has order $2$. Just as in the case of finite fields, $f(x,y)$ always has either an odd or an even factorization over $\R$.

If the form $f$ is negative definite, then there are no orbits
having invariant binary form $f(x,y)$. Indeed, in this case $r_1 = 0$ and the leading coefficient $f_0$ is negative.
Morever, the hyperelliptic curve $C$ with equation $z^2 = f(x,y)$ has no real points, and the
map $\Pic_{C/\R}(\R) \rightarrow \Br(\C/\R) = \Z/2\Z$ is surjective. The real divisor classes that are not represented by real divisors
have degrees congruent to $g-1$ modulo $2$.
When $g$ is even, the Jacobian $J(\mathbb R)$ is connected and every principal homogeneous space for $J$
is trivial. In particular, $J^1$ has real points (which are not represented by
real divisors of odd degree). When $g$ is odd, the real points of the Jacobian $J(\R)$ have two connected components, and $J^1$
is the unique nontrivial principal homogeneous space for $J$. The points in the connected component of $J(\R)$
are the real divisor classes of degree zero that are represented by real divisors.

If $f$ is not negative definite, then the element $f_0$ is a norm from $L^\TIMES$ to $\R^\times$. Hence rational orbits exist. When $r_1=0$, so $f$ is positive definite, there are two orbits if $n\equiv0$ (mod $4$) and there is only one orbit if $n\equiv2$ (mod $4$). In both cases, the real points of the hyperelliptic curve $C(\R)$ and its Jacobian $J(\R)$
are both connected and the orbits are all soluble.

If $r_1 > 0$, then the form $f$ is indefinite and the number
of orbits is $2^{r_1 - 2}$. These orbits are in bijection with the equivalence classes of sign assignments to the $r_1$ real linear factors of $f(x,y)$ subject to the condition that the product of the signs matches the sign of the leading coefficient of $f(x,y)$ and where two sign assignments are equivalent if they are exactly the negative of each other.
The hyperelliptic curve $C$ with equation $z^2 = f(x,y)$
has $m = r_1/2$ connected components in its real locus, and $J(\R)$ has $2^{m-1}$ connected components.
Since the subgroup $2J(\R)$ is equal to
the connected component of $J(\R)$,
it follows that $2^{m-1}$ of these rational orbits with invariant binary form $f$ are soluble.

The computation for the sizes of the stabilizers is similar to the finite field case. If $r_1=0$, then the size of the stabilizer is $2^{n/2}$ if $n\equiv0$ (mod $4$) and is $2^{n/2-1}$ if $n\equiv2$ (mod $4$). If $r_1>0$, then the size of the stabilizer is $2^{n/2+m-2}$ where again $m=r_1/2$.

\section{Global fields and locally soluble orbits}\label{sec:global}

In this section, we assume that $K$ is a global field of characteristic not 2. Let $f(x,y)$ be a binary form of degree $n=2g+2$ over $K$ with nonzero discriminant. Let $C:z^2=f(x,y)$ denote the associated hyperelliptic curve. Recall that an element $(A,B)$ of $K^2\otimes\Sym_2K^n$ (or its $(\SL_n/\mu_2)(K)$-orbit) with invariant binary form $f(x,y)$ is \textit{locally soluble} if the associated Fano variety $F(A,B)$ over $K$ has points over every completion $K_\nu$. We wish to determine when rational orbits and locally soluble orbits for the action of $(\SL_n/\mu_2)(K)$ on $K^2\otimes\Sym_2K^n$ exist. Theorem \ref{thm:with8equiv} gives a list of necessary and sufficient conditions for the existence of rational orbits over general fields. In this section, we assume that there exists a locally soluble two-cover of $J^1$ over $K$ and that $\Div^1(C)$ is locally soluble. The main result is that these two conditions are sufficient for the existence of a rational orbit and indeed a locally soluble orbit with invariant binary form $f(x,y)$. The proof will be cohomological in nature using Theorem \ref{thm:with8equiv}.

Recall the torsor $W[2]$ of $J[2]$ which consists of points $P\in J^1$ such that $2P=d$, where $d$ is the hyperelliptic class of $C$. The class of $W[2]$ in $H^1(K,J[2])$ maps to the class of $J^1$ in $H^1(K,J)[2]$. Since $J^1(K_\nu)$ is nonempty for all $\nu$, we see that a priori $W[2]$ lies in the $2$-Selmer subgroup $\Sel_2(J/K)$ of $H^1(K,J[2])$. Let $\pi:F_0\rightarrow J^1$ denote a locally soluble two-cover of $J^1$ over $K$. Let $F_0[4]$ denote the torsor of $J[4]$ consisting of points $x\in F_0$ such that $\pi(x)\in W[2]$. Then the class of $W[2]$ equals $m_2(F_0[4])$ where we recall that $m_2:H^1(K,J[4])\rightarrow H^1(K,J[2])$ is the map induced by multiplication by $2$ from $J[4]$ to $J[2]$. Since $F_0(K_\nu)$ is nonempty for all $\nu$, the class of $F_0[4]$ is in the $4$-Selmer subgroup $\Sel_4(J/K)$ of $H^1(K,J[4])$.

Conversely, suppose $C$ is any hyperelliptic curve over $K$ with
locally soluble $\Div^1(C)$ such that $W[2]$ is divisible by 2 in $\Sel_4(J/K).$ Then a locally soluble two-cover of $J^1$ over $K$ exists. Indeed, suppose $W[2]=m_2(F[4])$ for some $F[4]\in\Sel_4(J/K).$ Let $F$ denote the principal homogeneous space of $J$ whose class in $H^1(K,J)$ is the image of $F[4]$ in $H^1(K,J)[4]$. Then $2F = [J^1]$ and hence there exists a map $F\rightarrow J^1$ realizing $F$ as a two-cover of $J^1.$

\begin{theorem}\label{prop:lsol}
Suppose $C:z^2=f(x,y)$ is a hyperelliptic curve over a global field $K$ of characteristic not $2$ such that $C$ has a rational divisor of degree $1$ locally everywhere and such that $J^1$ admits a locally soluble two-cover over $K$ \emph{(}equivalently, $W[2]$ is divisible by $2$ in ${\rm{Sel}}_4(J/K)$\emph{)}. Then there exists $(A,B)\in K^2\otimes\Sym_2K^n$ with invariant binary form $f(x,y)$. That is, orbits for the action of $(\SL_n/\mu_2)(K)$ on $K^2\otimes\Sym_2K^n$ with invariant binary form $f(x,y)$ exist.
\end{theorem}

\begin{proof} Let $T=(\Res_{K'/K}\mathbb{G}_m)_{N=1}$ be the kernel of $J_{\frak m}\rightarrow J$ as in \eqref{eq:TJmJ}, where $K'=K[x]/(x^2-f_0)$. We will need the following properties about the cohomology of $T$:
\begin{enumerate}
\item $H^1(K_\nu,T) = K_\nu^\times/NK_\nu'^\times$ has exponent $2$ for any local completion $K_\nu$ of $K$;
\item $H^1(K,T)=K^\times/NK'^\times$ satisfies the local-global principle since $K'/K$ is cyclic when $K'$ is a field and $H^1(K,T)$ is trivial when $K'\simeq K\oplus K$;
\item $H^2(K,T) = \mbox{Br}(K')_{N=1}$ satisfies the local-global principle with respect to places of $K$;
\item The map $H^1(K_\nu,T)\rightarrow H^1(K_\nu,J_{\frak m})$ is injective for any local completion $K_\nu$ of $K$ since $\Div^1$ is locally soluble.
\end{enumerate}

Let $\phi,i,\delta$ be defined by the following diagram arising as part of the long exact sequence in Galois cohomology:
\begin{displaymath}
\xymatrix{
H^1(K,T)\ar[r]^{i}\ar[d]^{2}& H^1(K, J_{\frak m})\ar[r]^{\phi}\ar[d]^{2} & H^1(K, J)\ar[d]^{2}\ar[r]^{\delta} &H^2(K,T)\\
H^1(K,T)\ar[r]^{i}& H^1(K, J_{\frak m})\ar[r]^{\phi} & H^1(K, J)&}
\end{displaymath}
where the vertical maps are all multiplication by $2$. Let $[F]$ be a locally trivial class in $H^1(K,J)$ such that $2[F]=[J^1]$. By Theorem \ref{thm:with8equiv}, it suffices to show that the class $[J^1_{\frak m}]$ is divisible by $2$ in $H^1(K,J_{\frak m})$.

Since $[F]$ is locally trivial, its image under $\delta$ is also locally trivial. Since $H^2(K,T)$ has the local-global principle, it follows that $\delta([F]) = 0$ and so $[F]$ is in the image of $\phi$. Let $[F_{\frak m}]$ denote a class in $H^1(K,J_{\frak m})$ mapping to $[F]$ via $\phi$. Since $\phi([F_{\frak m}])$ is locally trivial, we see that $[F_{\frak m}]$ locally is in the image of $i$. Since $H^1(K_\nu,T)$ has exponent $2$ for every local completion of $K$, it follows that $2[F_{\frak m}]$ is locally trivial.

Now both $2[F_{\frak m}]$ and $[J^1_{\frak m}]$ are locally trivial and map to $[J^1]$ under $\phi$. We claim they are in fact equal. Indeed, their difference $2[F_{\frak m}]-[J^1_{\frak m}]$ is a locally trivial element of $H^1(K,J_{\frak m})$ mapping to $0$ under $\phi$. Hence there exists some $c\in H^1(K,T)$ such that $2[F_{\frak m}]-[J^1_{\frak m}]=i(c)$. Since the $\nu$-adic restrictions of $i$ are all injective, it follows that $c$ is locally trivial and hence trivial by the local-global principle of $H^1(K,T)$. This shows that $[J^1_{\frak m}]=2[F_{\frak m}]$ is divisible by $2$.
\end{proof}

Under the assumption that $\Div^1(C)$ is locally soluble, the existence of a locally soluble two-cover of $J^1$ is in fact equivalent to the existence of a locally soluble orbit for the action of $(\SL_n/\mu_2)(K)$ on $K^2\otimes\Sym_2K^n$. We will see that $\Sel_2(J/K)$ acts simply transitively on the set of locally soluble orbits. Therefore, every locally soluble two-cover of $J^1$ is isomorphic to the Fano variety $F(A,B)$ associated to the pencil of quadrics determined by some $(A,B)\in K^2\otimes\Sym_2K^n.$ This also proves Theorem \ref{2Selpar}.

\begin{theorem}\label{thm:lsoluble}
Suppose $C:z^2=f(x,y)$ is a hyperelliptic curve over a global field $K$ of characteristic not $2$ such that $\Div^1(C)(K_\nu)\neq\varnothing$ for all places $\nu$ of $K$. Then locally soluble orbits for the action of $(\SL_n/\mu_2)(K)$ on $K^2\otimes\Sym_2K^n$ with invariant binary form $f(x,y)$ exist if and only if $W[2]$ is divisible by~$2$ in $\Sel_4(J/K)$, or equivalently $J^1$ admits a locally soluble two-cover over $K$.
Furthermore, when these conditions are satisfied, the group $\Sel_2(J/K)$ acts simply transitively on the set of locally soluble orbits and this set is finite.
\end{theorem}

Before proving Theorem \ref{thm:lsoluble}, we note that the notion of locally soluble orbit is a tricky one. There could exist an integral binary quartic form $f(x,y)$ that has locally soluble orbits but no soluble orbits over $\Q$. For a specific example (suggested by John Cremona; see also \cite[\S 8.1]{Sk}), consider the elliptic curve $E$ defined by the equation $y^2 = x^3 - 1221$. This curve has trivial Mordell-Weil group $E(\Q) = 0$ and Tate-Shafarevich group isomorphic to $(\Z/4\Z)^2$. The binary quartic form $f(x,y) = 3x^4 - 12x^3y + 11xy^3 - 11y^4$ of discriminant $\Delta = -40252707 = -3^511^237^2$ corresponds to a class $b$ in the Tate-Shafarevich group of $E$ that is divisible by 2. Any of the elements $c$ of order $4$ in the Tate-Shafarevich group with $2c = b$ gives a locally soluble orbit with invariant binary form $f(x,y)$. The hyperelliptic curve $z^2=f(x,y)$ is locally soluble but has no global points; hence, by Theorem~\ref{thm:soluble}, there is no soluble orbit having invariant binary form $f(x,y)$.

There are also examples where rational orbits exist but there are no locally soluble orbits. For example, consider the binary quartic form $f(x,y) =  -x^4 + 2x^3y + 104x^2y^2 - 104xy^3 - 2764y^4$ of discriminant $\Delta = -2^8571$. The associated quartic field $L$ has discriminant $-2^4571 = -9136$ and ring of integers $\Z[\theta]$, where $\theta$ is a root of the polynomial $F(t) = t^4 - 2t^2 + 2t - 3$. Since $F(1) = -2$, $F(0) = -3$, and $F(-1) = -6$, the element $\theta^3 - \theta$ in $L^\times$ has norm $-6^2 \equiv -1 = f_0$. So there are orbits over $\Q$ with this invariant binary quartic form. On the other hand, the hyperelliptic curve $C:z^2=f(x,y)$ of genus one is a principal
homogeneous space of order $2$ for its Jacobian $E$, which is an elliptic curve with equation $y^2 + xy =  x^3 - x^2 - 929x - 10595$ and prime conductor $571$. This curve has trivial Mordell-Weil group $E(\Q) = 0$ and Tate-Shafarevich group isomorphic to $(\Z/2\Z)^2$. Hence $\text{Sel}_2(E/\Q)$ and $\text{Sel}_4(E/\Q)$ are both isomorphic to $(\Z/2\Z)^2$. The curve $C$ represents one of the nontrivial locally trivial principal homogeneous spaces for $E$. Since its class is not in the image of multiplication by $2$ from $\text{Sel}_4(E/\Q)$, there are no locally soluble orbits. (Thanks to John Cremona and Noam Elkies for help with computation in this example.)

\medskip
\noindent \textbf{Proof of Theorem \ref{thm:lsoluble}:}
Suppose locally soluble orbits with invariant binary form $f(x,y)$ exist. We prove first that $\Sel_2(J/K)$ acts simply transitively on the set of locally soluble orbits with invariant binary form $f(x,y)$. Indeed, suppose that $(A,B)$ is a rational pencil with Fano variety $F(A,B)$ and invariant binary form $f(x,y)$. Any other rational pencil $(A_1,B_1)$ with the same binary form corresponds to a class $c$ in $H^1(K, J[2])$ that is in the kernel of the composite map $\gamma: H^1(K,J[2]) \rightarrow H^1(K, \SL_n/\mu_2)\hookrightarrow H^2(K,\mu_2)$ (\cite[Proposition 1]{BG}). The map $\gamma$ is cup product with the class $W[2]\in H^1(K,J[2])$ (\cite[Proposition 10.3]{PSh}). Let $F(A_1,B_1)$ denote the Fano variety associated to the pencil $(A_1,B_1)$. Then one has, up to sign (cf. Footnote \ref{ft:sign}),
\begin{equation}\label{eq:FABc}
[F(A_1,B_1)] = [F(A,B)] + j(c),
\end{equation}
where $j$ denotes the natural map $H^1(K,J[2])\rightarrow H^1(K,J)[2]$ and the addition is taking place in $H^1(K,J).$ Since the subgroup $J(K_\nu)/2J(K_\nu)$ of $H^1(K_\nu,J[2])$ maps to the trivial class in $H^1(K_\nu,\SL_n/\mu_2)$ for all places $\nu$, the Hasse principle for the cohomology of the group $\SL_n/\mu_2$ shows that the subgroup $\Sel_2(J/K)$ of $H^1(K,J[2])$ also lies in $\ker\gamma.$ It is then clear from \eqref{eq:FABc} that if $(A,B)$ is locally soluble, then $c\in \Sel_2(J/K)$ if and only if $(A',B')$ is locally soluble. Hence $\Sel_2(J/K)$ acts simply transitively on the set of locally soluble orbits with invariant binary form $f(x,y)$. Since the $2$-Selmer group is finite, the set of locally soluble orbits with invariant binary form $f(x,y)$ is also finite. Moreover, if $(A,B)$ is locally soluble, then $F(A,B)$ gives a locally soluble two-cover of $J^1$ over~$K$.

We now consider the sufficiency of the existence of a locally soluble two-cover of $J^1$ for the existence of locally soluble orbits. Let $F$ denote the Fano variety corresponding to one rational orbit with invariant binary form $f(x,y)$. The existence of this rational orbit was the content of Theorem \ref{prop:lsol}. Let $F[4]$ denote the lift of $F$ to a torsor of $J[4]$ consisting of elements $x\in F$ such that $x+x+x+x=0$ in the group $X$ of four components defined in Theorem \ref{thm:pencil}. Let $\iota:H^1(K,J[2])\rightarrow H^1(K,J[4])$ denote the map induced from the inclusion of $J[2]$ inside $J[4]$. Then we have the following exact sequence:
\begin{equation}\label{eq:iotam2}
 H^1(K,J[2])\xrightarrow{\iota} H^1(K,J[4])\xrightarrow{m_2} H^1(K,J[2]).
\end{equation}
We need to show that there exists a class $c\in H^1(K, J[2])$ such that $c \cup W[2] = 0$ and $F[4] + \iota(c)\in \Sel_4(J/K).$ Let $d_0$ be a class in $\Sel_4(J/K)$ such that $W[2] = m_2(d_0).$ Since $m_2(F[4]-d_0)=W[2]-W[2]=0$, there exists an element $c_0\in H^1(K,J[2])$ such that $\iota(c_0)=F[4]-d_0$ by the exact sequence \eqref{eq:iotam2}. Then it suffices to show that
\begin{equation}\label{eq:cupproduct}
c_0\cup W[2] = 0.
\end{equation}
For ease of notation, we denote the above cup product by $e_2(c_0,W[2])$ since the cup product is induced from the Weil pairing $e_2$ on $J[2]$. Since $d_0\in \Sel_4(J/K)$ is isotropic with respect to $e_4$, we have
$$e_2(c_0,W[2]) = e_4(F[4]-d_0, d_0) = e_4(F[4],d_0).$$
Fix a place $\nu$ and denote by $F[4]_\nu,$ $d_{0,\nu},$ $e_{4,\nu}$ the $\nu$-adic restrictions. Pick any $D_\nu \in J^1(K_\nu)$. Since $F$ arises from a pencil of quadrics, we define
$$F[2]^{D_\nu} = \{x\in F: x + x = D_\nu\}.$$
The image of this torsor of $J[2]$ in $H^1(K_\nu, J[4])$ is the torsor
$$F[4]^{2D_\nu - d} = \{x\in F: x + x + x + x = 2D_\nu - d\},$$
where $d$ denotes the hyperelliptic class as before. Therefore, as elements of $H^1(K_\nu, J[4])$, we have
$$F[4]_\nu - \iota_\nu(F[2]^{D_\nu}) = \delta_{4,\nu}(2D_\nu - d),$$
where $\delta_{4,\nu}$ is the Kummer map $J(K_\nu)/4J(K_\nu)\rightarrow H^1(K_\nu, J[4])$ and $\iota_\nu$ is the $\nu$-adic restriction of $\iota$. Since $d_0\in \Sel_4(K, J),$ we see that $d_{0,\nu}$ is in the image of $\delta_{4,\nu}.$ Since $J(K_\nu)/4J(K_\nu)$ is isotropic with respect to $e_{4,\nu}$, we have
\begin{equation}\label{eq:e_2FW}
e_{4,\nu}(F[4]_\nu,d_{0,\nu}) = e_{4,\nu}(\iota_\nu(F[2]^{D_\nu}),d_{0,\nu}) = e_{2,\nu}(F[2]^{D_\nu},W[2]_\nu).
\end{equation}
Choosing a different $D_\nu\in J^1(K_\nu)$ changes $F[2]^{D_\nu}$ by an element of $J(K_\nu)/2J(K_\nu)$. As $J(K_\nu)/2J(K_\nu)$ is isotropic with respect to $e_2$, the value of $e_{2,\nu}(F[2]^{D_\nu},W[2]_\nu)$ does not depend on the choice of $D_\nu.$
Theorem \ref{thm:lsoluble} then follows from the following general lemma.

\begin{lemma} Suppose $K$ is any local field of characteristic not $2$. Let $f(x,y)$ be a binary form of degree $2g+2$ with nonzero discriminant such that the associated hyperelliptic curve $C:z^2=f(x,y)$ satisfies $\Div^1(C)(K)\neq\varnothing.$ Suppose there is a rational orbit for the action of $(\SL_n/\mu_2)(K)$ on $K^2\otimes\Sym_2K^n$ with invariant binary form $f(x,y)$, and let $F$ denote the associated Fano variety. Then
\begin{equation}\label{eq:e2cup}
e_2(F[2],W[2]) = 0,
\end{equation}
where $F[2]$ denotes any lift of $F$ to a torsor of $J[2]$ using a point of $J^1(K).$
\end{lemma}

\begin{proof}
The first key point is that if \eqref{eq:e2cup} holds for one rational orbit, then it holds for any rational orbit with the same invariant binary form. Indeed, if $F'$ denotes the torsor of $J$ coming from a different orbit, then $F' - F \in \ker\gamma,$ where $\gamma:H^1(K,J[2])\rightarrow H^2(K,\mu_2)$ is cup product with $W[2]$.
In other words, $e_2(F'-F,W[2]) = 0.$ Hence $e_2(F[2],W[2]) = e_2(F'[2],W[2]).$

The second key point is that since $\Div^1(C)(K)\neq\varnothing,$ there exists a soluble orbit by Theorem~\ref{thm:soluble}. Let $F$ denote the corresponding torsor arising from this soluble pencil. Then $F[2]\in J(K)/2J(K)$ and hence $e_2(F[2],W[2])=0.$
\end{proof}

\noindent
This completes the proof of Theorem~\ref{thm:lsoluble}.\hfill$\Box$

\medskip

We conclude by
remarking that the natural generalization of the \emph{fake $2$-Selmer set}
$\Sel_{2,\hspace{1pt}{\rm fake}}(C)$ of $C$ (\cite{BruSto}), namely
the \emph{fake $2$-Selmer set $\Sel_{2,\hspace{1pt}{\rm fake}}(J^1)$} of $J^1$, is in
natural bijection with the set of locally soluble orbits for the group
$(\SL^{\pm}_n/\mu_2)(K),$ where $\SL^{\pm}_n$ denotes as before the
subgroup of elements of $\GL_n$ with determinant $\pm1$.  Using the
group $\SL_n$ instead of $\SL^{\pm}_n$ allows us to ``unfake'' this
fake Selmer set (cf.\ \cite{StovL}).

\section{Existence of integral orbits}\label{integral}

The purpose of this section is to prove Theorem~\ref{2SelparZ}. More precisely, we prove:

\begin{theorem}\label{locsol}
Assume that $n\geq 2$ is even. Let $f(x,y)$ be a binary form of degree $n=2g+2$ with coefficients in $16^n\Z$ such that the hyperelliptic curve $C:z^2=f(x,y)$ has locally soluble $\Div^1$. Then every locally soluble orbit for the action of $(\SL_n/\mu_2)(\Q)$ on $\Q^2\otimes\Sym_2\Q^n$ with invariant binary form $f(x,y)$
has an integral representative, i.e., a representative in $\Z^2\otimes\Sym_2\Z^n.$
\end{theorem}

By Theorem \ref{orbit2} with $D=\Z$ and $\Z_p$, it suffices to find a representative over $\Z_p$ for every soluble orbit over $\Q_p$ with $f(x,y) \in \Z_p[x,y]$ since an ideal can be defined by giving its localization and its norm is always principal since $\Z$ is a PID. We begin by recalling from \cite[\S2]{B} the construction of an integral orbit associated to a rational point on $C$, or a $p$-adically integral orbit associated to a $p$-adic point on $C$. For this we recall some of the notations in Section \ref{dedorbits}. Without loss of generality, we may assume $f_0\neq0$. (By our convention, $C$ being a hyperelliptic curve is equivalent to $\Delta(f)\neq0.$) Write $f(x,1)=f_0g(x)$ and let $L=\Q_p[x]/g(x)$ be the corresponding \'{e}tale algebra of rank $n$ over $\Q_p$. For $k = 1,2,\ldots,n-~1$, there are integral elements
$$\zeta_k = f_0 \theta^k + f_1 \theta^{k-1} + \cdots + f_{k-1} \theta$$
in $L$. Let $R_f$ be the free $\Z_p$-submodule of $L$ having $\Z_p$-basis $\{ 1,\zeta_1,\zeta_2, \ldots,\zeta_{n-1}\}$. For $k=0,1,\ldots, n-~1$, let $I_f(k)$ be the free $\Z_p$-submodule of $L$ with basis $\{ 1, \theta, \theta^2,\ldots, \theta^{k}, \zeta_{k+1}, \ldots, \zeta_{n-1}\}$. By Theorem~\ref{orbit2}, an integral orbit is an equivalence class of triples $(I,\alpha,s)$ where $I$ is an ideal of $R_f$, $\alpha\in L^\times$, and $s\in K^\times$, such that $I^2\subset\alpha I_f(n-3),$ $N(I)=s\Z_p$, and $N(\alpha)=s^2f_0^{n-3}.$ The rational orbit is given by the equivalence class of the pair $(\alpha,s).$

By a change of variable, we may assume that we have an integral point $P = (0,1,c)$ on the curve $z^2 = f(x,y)$ over $\Z_p$, so that the coefficient $f_n = c^2$ is a square. Then set $\alpha=\theta$, and we have
\begin{equation}
\theta I_f(n-3) = \Span_{\Z_p}\{
c^2,\theta,\theta^2,\ldots,\theta^{n-2},f_0\theta^{n-1}
\}.
\end{equation}
Let $I=\Span_{\Z_p}\{
c,\theta,\theta^2,\ldots,\theta^{(n-2)/2},\zeta_{n/2},\ldots,\zeta_{n-1}\}$.
Then it is easy to check that $I$ is an ideal of $R_f$, $I^2\subseteq \alpha I_f(n-3)$, and
$$N(I)^2 = N(\theta)N(I_f(n-3))=[c/f_0^{(n-2)/2}]^2\Z_p.$$
Let $s=\pm c/f_0^{(n-2)/2}$ be such that $(\alpha,s)$ corresponds to the rational orbit determined by $P$. The triple $(I,\alpha,s)$ gives an integral orbit representing the soluble orbit given by $P$ in $J^1(\Q_p)/2J(\Q_p)$. We note that this association of an integral orbit to a $\Q$-rational point, and the paucity of integral orbits, was the key to the arguments of \cite{B} showing that rational points are rare.

Given one such $f(x,y)=f_0x^n + f_1x^{n-1}y + \cdots + f_ny^n$ with coefficients in $16\Z$, then $2^{4i}\mid f_0^{i-1}f_i$ for $i=1,\ldots,n$. Therefore, Theorem \ref{locsol} follows from the following proposition where the assumption on the coefficients is asymmetrical in contrast to Theorem \ref{locsol}:

\begin{proposition}\label{locsolasy}
Assume that $n\geq 2$ is even. Let $f(x,y)=f_0x^n + f_1x^{n-1}y + \cdots+f_ny^n$ be a binary form of degree $n=2g+2$ satisfying $f_0 \neq 0$ and $2^{4i}\mid f_0^{i-1}f_i$ for $i=1,2,\ldots,n$ such that the hyperelliptic curve $C:z^2=f(x,y)$ has locally soluble $\Div^1$. Then every locally soluble rational orbit for the action of $(\SL_n/\mu_2)(\Q)$ on $\Q^2\otimes\Sym_2\Q^n$ with invariant binary form $f(x,y)$ has an integral representative.
\end{proposition}

\begin{proof}
We work over $\Z_p$ and give an explicit construction of the ideal $I$, in a manner similar to the one-point case shown above (cf. \cite[\S2]{B}) and the corresponding statements in \cite[Proposition 8.2]{BG1} and \cite[Proposition 2.9]{SW}. There are several important differences due to $f_0$ not being $1$.

Define $g(x,y) = x^n + f_1x^{n-1}y + f_0f_2x^{n-2}y + \cdots + f_0^{n-1}f_ny^n.$ Then $g(f_0x,y) = f_0^{n-1}f(x,y)$ and so $(f_0\theta,1)$ is a root of $g$. The condition $2^{4i}\mid f_0^{i-1}f_i$, which is nontrivial only when $p=2$, implies that if $a\in \Q_p$ is non-integral, then $a - f_0\theta \in L^\times$ lies in $L^{\times2}\Q_p^\times.$

We claim that it suffices to consider classes in $J^1(\Q_p)/2J(\Q_p)$ that can be represented by a Galois-invariant divisor of the form
\begin{equation}\label{eq:Dm}
D  = (P_1) + (P_2) + \cdots + (P_m) - D^*,
\end{equation}
such that:
\begin{enumerate}
\item The points $P_i = (a_i,b_i,c_i)$ are non-Weierstrass and non-infinite;
\item The effective divisor $D^*$ is supported on points above $\infty$;
\item The positive integer $m$ is odd with $m\leq g+1$;
\item For every $i=1,\ldots,m$, scale $a_i,b_i,c_i$ so that $b_i=1$. Then $f_0a_i$ is integral and the $a_i$'s are distinct.
\end{enumerate}

Since $\Div^1(C)(\Q_p)\neq\varnothing$, every $\Q_p$-rational divisor class can be represented by a rational divisor by Proposition \ref{pic}. By \cite[Lemma 3.8]{W2}, every class in $J^1(\Q_p)/2J(\Q_p)$ has the desired form satisfying conditions 1, 2, 3 except for the oddness of $m$. As remarked above, if $f_0a_i$ is not integral, then $f_0a_i - f_0\theta\in L^{\times2}\Q_p^\times$ and so is $a_i - \theta$. Removing all the points $P_i$ with $f_0a_i$ non-integral gives a rational divisor $D'$ that has the same image as $D$ via the $x-T$ map from $J^1(\Q_p)/2J(\Q_p)$ to $(L^\times/(L^{\times2}\Q_p^\times))_{N=f_0}$. By Theorem \ref{orbit2}, a triple $(I,\alpha,s)$ exists for $D$ if and only if it exists for $D'$. We may impose the condition that the $a_i$'s are all distinct since we are working modulo $2J(\Q_p)$.

We now show that we only need to consider the case $m$ odd. If $m$ is even, then it forces $f_0$ to be a square and so the points at infinity are rational. Let $\infty$ denote one of them. Since $D$ has degree $1$, we see that the degree of $D^*$ is odd. Let $x_0$ be an element of $\Z_p$ to be chosen later and consider the change of coordinate $(x,y) \mapsto (x - x_0y,y) \mapsto (-y,x-x_0y).$ Let $\w{f}(x,y)$ denote the new binary form, which is $\SL_2(\Z_p)$ equivalent to $f(x,y)$. Let $\w{C}$ and $\w{J}$ denote the new hyperelliptic curve and its Jacobian. Let $Q_1,\ldots,Q_m,Q_0\in \w{C}$ denote the images of $P_1,\ldots,P_m,\infty$. We pick $x_0$ so that none of the points $Q_i$ are Weierstrass for $\w{C}$. The divisor $D$ becomes the divisor $\w{D} = (Q_0) + \cdots + (Q_m) - \mbox{ points at infinity}$, up to $2\w{J}(\Q_p)$. If integral orbits exist for $\w{D}$, then applying the inverse of the above $\SL_2(\Z_p)$ transformation gives the desired integral orbits for $D$. There are $m+1$ non-Weierstrass and non-infinite points in $\w{D}$ and so we are done if $m\leq g.$

Suppose now $m=g+1$ is even. Let $\w{R}(x)$ be a polynomial of degree at most $g+1$ such that $\w{R}(\w{a_i}) = \w{c_i}$ where $Q_i = (\w{a_i},1,\w{c_i})$ for each $i=0,\ldots,g+1$. Then $\w{f}(x,1) - \w{R}(x)^2$ has degree at most $2g+2$ and vanishes at $\w{a}_0,\ldots,\w{a}_{g+1}$. So it has at most $g$ other roots. This shows that $\w{D}$ is rationally equivalent to a divisor of the form $(R_1) + \cdots + (R_{m'}) - \mbox{ points at infinity}$, with $m'\leq g$. If $m'$ is odd, then we are done. If $m'$ is even, then since $m'$ is now at most $g$, we may apply the above construction to obtain a divisor $D''$ of the form $(S_1) + \cdots + (S_{m'+1}) - \mbox{ points at infinity}$, such that the existence of integral orbits is equivalent for $D''$, $\w{D}$ and $D$.

Suppose now $D$ is a divisor of the form \eqref{eq:Dm} satisfying conditions 1--4. Define $P(x) = (x - f_0a_1)\cdots(x - f_0a_m).$ By our assumption on the integrality of $f_0a_i$, $P(x)$ is an integral polynomial. Write $\alpha_0 = (a_1 - \theta)\cdots(a_m - \theta).$ Then $P(f_0\theta) = - f_0^m\alpha_0.$ Next define $R(x)$ to be a polynomial of degree at most $m-1$ so that $R(f_0a_i) = f_0^{n/2}c_i$ for each $i=1,\ldots,m$. Then $R(x)^2 - f_0g(x,1)$ vanishes at $f_0a_1,\ldots,f_0a_m.$ So there exists an integral polynomial $h(x)$ such that $R(x)^2 - f_0g(x,1) = P(x)h(x)$. Note we have $R(f_0\theta)^2 = P(f_0\theta)h(f_0\theta).$

Suppose first $R(f_0x)$ is an integral polynomial. Then we set $I_D$ to be the following $R_f$-submodule of $L$:
$$I_D = \langle f_0^{2m}R(f_0\theta), P(f_0\theta)I_f(\frac{n-3-m}{2})\rangle.$$
Computing its square gives:
\begin{eqnarray*}
I_D^2 &=& P(f_0\theta)\cdot\langle f_0^{4m}h(f_0\theta), f_0^{2m}R(f_0\theta)I_f(\frac{n-3-m}{2}), P(f_0\theta)I_f(n-3-m)\rangle\\
&=& P(f_0\theta)f_0^m\cdot\langle f_0^{3m}h(f_0\theta), f_0^mR(f_0\theta)I_f(\frac{n-3-m}{2}), (\theta-a_1)\cdots(\theta-a_m)I_f(n-3-m)\rangle\\
&\subset& f_0^{2m}\alpha_0 I_f(n-3).
\end{eqnarray*}
The last containment follows from computing the degrees of $h$ and $R$. (When $m=1$, one checks directly that $h\in I(n-3)$.) We then set $\alpha = f_0^{2m}\alpha_0$ to get $I^2 \subset \alpha I(n-3)$.

To compute the norm of $I_D$, we use a specialization argument. Let $\RR$ denote the ring $$\RR=\Z_p[f_0,\ldots,f_n,a_1,\ldots,a_m][\sqrt{f(a_1,1)},\ldots,\sqrt{f(a_m,1)}],$$
where $f(x,1)=f_0x^n + f_1x^{n-1} + \cdots + f_n.$ Write $c_i=\sqrt{f(a_i,1)}$ for each $i=1,\ldots,m$. Inside $\RR[x]/f(x,1),$ we define $\zeta_1,\ldots,\zeta_{n-1}$ as before and denote the corresponding $R_f,I_f,I_D$ by $\RR_f,\II_f,\II_D.$ One has the notion of $N\II_D$ as an $\RR$-submodule of the fraction field of $\RR$.

We claim that $N\II_D$ is the principal ideal generated by $s = c_1\cdots c_mf_0^{nm-(n-3+m)/2}.$ Specializing to particular $f_0,\ldots,f_n,a_1,\ldots,a_m$ then completes the proof.
We prove this claim by first inverting $f_0$. In this case, the result follows from \cite[Proposition~8.5]{BG}. Next we localize at $(f_0)$ to check that the correct power of $f_0$ is attained. Since every ideal is invertible now, it suffices to show that $\II_D^2 = \alpha\II_f^{n-3}$ which follows from the statements
\begin{equation}\label{eq:thetaminusa}
(\theta - a_i)(\RR_f)_{(f_0)} = (\II_f)_{(f_0)}
\end{equation}
for $i=1,\ldots,m$. To prove \eqref{eq:thetaminusa}, note that the containment $\subset$ is clear since $\theta-a_i\in\II_f$; equality follows because they have the same norm. We now give another more explicit proof of \eqref{eq:thetaminusa}. Note that it remains to show that $1\in(\theta-a_i)(\RR_f)_{(f_0)}.$ Consider the polynomial $h_i(t) = (f(t,1)-c_i^2)/(t-a_i).$ By definition $h_i(\theta)(\theta-a_i) = -c_i^2.$ Moreover, writing out $h_i(t)$ explicitly, one sees that
$$h_i(\theta) = \zeta_{n-1} + a_i\zeta_{n-2} + a_i^2\zeta_{n-3} + \cdots + a_i^{n-2}\zeta_1 + h_i(0)\in \RR_f.$$
This shows that $c_i^2\in (\theta - a_i)(\RR_f)_{(f_0)}$, and hence $1\in(\theta-a_i)(\RR_f)_{(f_0)}$ since $c_i$ is a unit in $(\RR_f)_{(f_0)}.$

We now deal with the case when $R(f_0x)$ is not integral. The rational function $y - R(f_0x)/f_0^{n/2}$ vanishes at $P_1,\ldots,P_m$ which prompts us to consider the divisor $\mbox{div}(y - R(f_0x)/f_0^{n/2})$, which amounts to studying the roots of $j(x) = f(x,1) - R(f_0x)^2/f_0^n$. Now $j(x)$ is a polynomial of degree $n$ with leading coefficient $f_0$ since the degree of $R^2$ is at most $2m-2<n$. Since $R(f_0x)$ is not integral, $j(x)$ has a coefficient of valuation strictly less than $-n\nu_p(f_0)$, where $\nu_p$ denotes the $p$-adic valuation. Then $j(x)$ has at least $n-(2m-2)$ roots with valuation less than $-\frac{n+1}{n}\nu_p(f_0)$ as seen from its Newton polygon. In other words, $j(x)$ has at least $n-(2m-2)$ roots $a_i^*$ such that $f_0a_i^*$ is not integral. These roots will then give a divisor that is divisible by $2$ in $J(\Q_p)$. Since $j(x)$ vanishes at the $x$-coordinates of $P_1,\ldots,P_m$, we see that it has at most $m-2$ other roots $a$ such that $f_0a$ is integral. Hence $\mbox{div}(y - R(f_0x)/f_0^{n/2})-D$ has the form $D' + E$ where $D'$ has the form \eqref{eq:Dm} with $m$ replaced by $m'\leq m-2$ and where $E\in 2J(\Q_p)$. If $m'$ is even, then as we have shown above, there exists a divisor $D''$ of the form \eqref{eq:Dm} with $m'+1<m$ non-Weierstrass non-infinite points. The proof now concludes by induction on $m$. Once $m = 1$, the polynomial $R(f_0x)$ is integral.
\end{proof}

\bigskip

In certain cases, we may show that a soluble rational orbit has a unique integral representative up to the action of $(\SL_n/\mu_2)(\Z_p)$:

\begin{proposition}\label{int=sol}
Let $p$ be any odd prime, and let $f(x,y)\in\Z_p[x,y]$ be a binary form of even degree $n$ such that $p^2\nmid\Delta(f)$ and $f_0\neq0$. Let $C$ denote the hyperelliptic curve $z^2 = f(x,y)$. Suppose that $\Div^1(C)(\Q_p)\neq\varnothing.$ Then the $(\SL_n/\mu_2)(\Z_p)$-orbits on $\Z_p^2\otimes\Sym_2\Z_p^n$ with invariant binary form $f(x,y)$ are in bijection with soluble $(\SL_n/\mu_2)(\Q_p)$-orbits on $\Q_p^2\otimes\Sym_2\Q_p^n$ with invariant binary form $f(x,y)$. Furthermore, if $(A,B)\in \Z_p^2\otimes\Sym_2\Z_p^n$ with invariant binary form $f(x,y)$, then $\Stab_{(\SL_n/\mu_2)(\Z_p)}(A,B)=\Stab_{(\SL_n/\mu_2)(\Q_p)}(A,B)$.
\end{proposition}

\begin{proof}
As noted earlier, we only need to focus on the pair $(I,\alpha)$. The condition $p^2\nmid \Delta(f)$ implies that the order $R_f$ is maximal and that the projective closure $\CC$ of $C$ over ${\rm{Spec}}(\Z_p)$ is regular. By Theorem~\ref{thm:soluble}, the assumption that $\Div^1(C)(\Q_p)\neq\varnothing$ implies that soluble $\Q_p$-orbits with invariant binary form $f(x,y)$ exist. Since $p$ is odd, the $p$-adic version of Proposition~\ref{locsolasy} implies that $\Z_p$-orbits with invariant binary form $f(x,y)$ exist. Therefore, by Remark \ref{rmk:maximal}, the set of equivalence classes of pairs $(I,\alpha)$ is nonempty and is in bijection with $(R_f^\times/(R_f^{\times2}\Z_p^\times))_{N\equiv 1}.$  Since the special fiber of $\mathcal{C}$ is geometrically reduced and irreducible, the Neron model $\J$ of its Jacobian $J_{\Q_p}$ is fiberwise connected (\cite[\S9.5 Theorem 1]{BLR}) and its
2-torsion $\J[2]$ is isomorphic to
$(\text{Res}_{R/\Z_p}\mu_2)_{N=1}/\mu_2.$ We have via \'{e}tale cohomology (\cite[Proposition 2.11]{SW}) that
$$\J(\Z_p)/2\J(\Z_p)\simeq (R_f^\times/(R_f^{\times2}\Z_p^\times))_{N\equiv 1}.$$
The N\'eron mapping property implies that
$\J(\Z_p)/2\J(\Z_p)=J(\Q_p)/2J(\Q_p).$

For the stabilizer statement, we have $L^{\times}[2] = R_f^{\times}[2]$ which suffices when $n\equiv 2\pmod{4}$. When $n\equiv 0\pmod{4}$, the exact sequence \eqref{eq:maximalstab} implies that it remains to compare
$(L^{\times2}\cap \Q_p^\times)/\Q_p^{\times2}$ and $(R_f^{\times2}\cap
\Z_p^\times)/\Z_p^{\times2}$. These two groups are nontrivial only when $L$ contains a quadratic extension $K'$ of $\Q_p$. Since $p^2\nmid\Delta(f)$ and $n\geq 4$, the discriminant of the extension $K'/\Q_p$ cannot be divisible by $p$. Hence $K'=\Q_p(\sqrt{u})$ can only be the unramified quadratic extension of $\Q_p$. In other words, $u\in\Z_p^\times.$ Hence in this case $(L^{\times2}\cap \Q_p^\times)/\Q_p^{\times2}$ and $(R_f^{\times2}\cap \Z_p^\times)/\Z_p^{\times2}$ both are equal to the group of order 2 generated by the class of $u$.
\end{proof}

\section{The number of irreducible integral orbits of bounded height}\label{counting}

Let $V=\Sym_2(W^*)\oplus\Sym_2(W^*)$ be the scheme of pairs of symmetric bilinear forms on $W$.
Define the height $H(v)$ of an element $v\in V(\Z)$ to be the height of its invariant binary form. We say that $v\in V(\Z)$ is \emph{irreducible} if its invariant binary form has nonzero discriminant.
In \cite[\S4]{B}, the asymptotic number of irreducible
$\SL_n^{\pm1}(\Z)$-orbits on $V(\Z)$ having height less than $X$ was
determined, and also the asymptotic number of such orbits whose invariant binary
forms satisfy any finite set of congruences.  
The same computation applies also with $G=\SL_n/\mu_2$ in place of $\SL_n^{\pm1}$. We assume henceforth that $n$ is even.

To state this counting result precisely, recall from the discussion of Section \ref{sec:RandCorbits} that we may naturally partition the
set of elements in $V(\R)$ with $\Delta\neq 0$ and whose invariant
binary form is not negative definite into $\sum_{m=0}^{n/2}
r(m)$ components, which we denote by $V^{(m,r)}$
for $m=0,1,\ldots,n/2$ and $r=1,\ldots,r(m)$ where: $r(m)=2^{2m-2}$ if
$m\geq1$; \;$r(0)=2$ if $n\equiv0$ (mod $4$); \,and $r(0)=1$ if $n\equiv2$ (mod $4$). A very similar partition is used in \cite[\S 4.1.1]{B}.

For a given value of~$m$, the component $V^{(m,r)}$ of $V(\R)$ maps to
the component $I(m)$ of nonnegative definite binary $n$-ic forms in $\R^{n+1}$ having
nonzero discriminant and $2m$ real linear factors.
Let $\FF^{(m,r)}$ denote a fundamental domain for the action of $G(\Z)$ on $V^{(m,r)}$,
and set
\[c_{m,r} = {\Vol(\FF^{(m,r)}\cap \{v\in V(\R):H(v)<1\})};\]
here $\Vol$ denotes the Euclidean measure on $V(\R)$. The number of $r$'s that correspond to orbits soluble at $\R$ is $\#(J^{1}(\R)/2J(\R))$ where $J$ denotes the Jacobian of a hyperelliptic curve $z^2=f(x,y)$ with $f(x,y)\in I(m)$. The size of this quotient does not depend on the choice of $f(x,y)\in I(m).$
Then from \cite[Theorems~9 and 17]{B}, we obtain the following counting result:

\begin{theorem}\label{thmcount}
 Fix $m,r$. Suppose $S$ is a $G(\Z)$-invariant subset of $V(\Z)^{(m,r)}:=V(\Z)\cap V^{(m,r)}$
 defined by finitely many congruence conditions modulo prime powers. Let $N(S;X)$
  denote the number of $G(\Z)$-equivalence classes of
  elements $v\in S$ satisfying $H(v)<X$. Then
$$N(S;X)=c_{m,r}\cdot
  \prod_{p} \nu_p(S)\cdot X^{n+1}+o(X^{n+1}),$$
where $\nu_p(S)$ denotes the $p$-adic density of $S$ in $V(\Z)$.
\end{theorem}

\section{Sieving to locally soluble orbits}\label{sec:sieve}

Since local solubility is defined by infinitely many congruence conditions, we need a weighted version of
Theorem \ref{thmcount} in which we allow weights to be defined by certain
infinite sets of congruence conditions. The technique for proving such
a result involves using Theorem \ref{thmcount} to impose more and more
congruence conditions.

To describe which weight functions on $V(\Z)$ are allowed, we need the
following definition:
\begin{defn}
  \emph{Suppose $U=\A^M$ is some affine space. A function $\phi:U(\Z)\to[0,1]$ is said to be} defined by congruence
  conditions \emph{if there exist local functions $\phi_p:U(\Z_p)\to[0,1]$ satisfying the following conditions:
  \begin{enumerate}
  \item For all $v\in U(\Z)$, the product $\prod_p\phi_p(v)$ converges to $\phi(v)$.
  \item For each prime $p$, the function $\phi_p$ is locally constant
    outside some ($p$-adically) closed subset of $U(\Z_p)$ of measure $0$.
  \item The $p$-adic integral $\displaystyle\int_{U(\Z_p)}\phi_p(v)dv$ is nonzero.
  \end{enumerate}
  A subset $U'$ of $U(\Z)$ is said to be} defined by congruence conditions \emph{if its characteristic function is defined by congruence conditions.}

\end{defn}

\noindent
Then we have the following theorem, which follows from Theorem \ref{thmcount} via a
sifting argument just as in \cite[\S2.7]{BS}.
\begin{theorem}\label{thinfcong}
  Let $\phi:V(\Z)\to[0,1]$ be a $G(\Z)$-invariant function that is defined
  by congruence conditions via local functions
  $\phi_p:V(\Z_p)\to[0,1]$. Fix $m,r$. Let $S$ be a $G(\Z)$-invariant subset of $V(\Z)^{(m,r)}$
 defined by congruence conditions. Let
  $N_\phi(S;X)$ denote the number of
  $G(\Z)$-equivalence classes of irreducible elements $v\in S$ having height bounded by $X$, where
  each equivalence class $G(\Z) v$ is counted with weight
  $\phi(v)$. 
Then
$$
N_{\phi}(S;X)\leq c_{m,r}X^{n+1}\prod_p\int_{v\in V(\Z_p)}\phi_p(v)dv+o(X^{n+1}).
$$
\end{theorem}

Identify the scheme of all binary $n$-ic forms over $\Z$ with $\A_\Z^{n+1}$ and let $F_0$ denote the set of all integral binary forms of degree $n$. If $F$ is a subset of $F_0$, denote by $F(\F_p)$ the reduction modulo $p$ of the $p$-adic closure of $F$ in $\A_\Z^{n+1}(\Z_p)$.

\begin{defn}\label{def:large} {\em A subset $F$ of $F_0$ is {\it large} if the following conditions are satisfied:
\begin{enumerate}
\item It is defined by congruence conditions. 
\item There exists a subscheme $S_0$ of $\A_\Z^{n+1}$ of codimension at least 2 such that for all but finitely many~$p$, we have $F_0(\F_p)\backslash F(\F_p)\subset S_0(\F_p).$
\end{enumerate}
We identify hyperelliptic curves with their associated binary forms. We say that a family of hyperelliptic curves $z^2 =f(x,y)$ is {\it large} if the set of binary forms $f(x,y)$ appearing is large.}
\end{defn}
As an example, the subset $F_1$ of $F_0$ consisting of binary $n$-ic forms $f(x,y)$ such that the corresponding hyperelliptic curves $C$ given by $z^2 = f(x,y)$ have locally soluble $\Div^1$ is large. The set $F_2\subset F_1$ of integral binary $n$-ic forms such that the corresponding hyperelliptic curves are locally soluble is also large. These two statements follow from \cite[Lemma 15]{PS}. Our aim is to prove the analogue of Theorem~\ref{count} for all large families of hyperelliptic curves whose associated binary forms are contained in $F_1$.

Let $F$ be a large subset of $F_0$ contained in $F_1$. Since the curves $z^2=f(x,y)$ and $z^2 = \kappa^2f(x,y)$ are isomorphic over $\Q$, where $\kappa$ is the constant in Theorem~\ref{2SelparZ}, we assume without loss of generality that the coefficients of every $f(x,y)$ in $F$ lie in $\kappa^2\Z$. To prove Theorem \ref{count}, we
need to weigh each locally soluble element $v\in V(\Z)$ whose invariant binary form is in $F$ by the reciprocal of the number of
$G(\Z)$-orbits in the $G(\Q)$-equivalence class of $v$ in
$V(\Z)$. However, in order for our weight function to be defined by
congruence conditions, we instead define the following weight function
$w:V(\Z)\to [0,1]$:
\begin{equation}\label{eqdefm}
  w(v):=
  \begin{cases}
    \Bigl(\displaystyle\sum_{v'}\frac{\#\Stab_{G(\Q)}(v')}{\#\Stab_{G(\Z)}(v')}\Bigr)^{-1} \qquad&\text{if $v$ is locally soluble with invariant binary form in $F$,}\\[.1in]
    \qquad\qquad 0 \qquad&\text{otherwise},
  \end{cases}
\end{equation}
where the sum is over a complete set
of representatives for the action of $G(\Z)$ on the
$G(\Q)$-equivalence class of $v$ in $V(\Z)$.
We then have the following theorem:
\begin{theorem}\label{thmisimp}
Let $F$ be a large subset of $F_0$ contained in $F_1$. Moreover,
suppose that the coefficients of every $f(x,y)\in F$ lie in $16^n\Z$. Then
\begin{equation}\label{eqtheoremimpone}
\sum_{\substack{C\in F\\H(C)< X}}\#\Sel_2(J^1)\leq\sum_{m=0}^{n/2}\sum_{\,r \,{\rm soluble}}N_w(V(\Z)_F^{(m,r)};X)+o(X^{n+1}),
\end{equation}
where $V(\Z)_F^{(m,r)}$ is the set of all elements in $V(\Z)^{(m,r)}$ whose invariant binary forms lie in $F$, and ``$r$~soluble'' is short for ``every element of $V(\Z)_F^{(m,r)}$ is soluble over~$\R$''.
\end{theorem}

\begin{proof} By Theorems \ref{thm:lsoluble} and \ref{locsol}, the left hand side is equal to the number of $G(\Q)$-equivalence classes of elements in $V(\Z)$ that are locally soluble, have invariant binary forms in $F$, and have height bounded by $X$. Given a locally soluble element $v\in V(\Z)$ with invariant binary form in $F$,
let $v_1,\ldots, v_k$ denote a complete set of representatives for the action of $G(\Z)$
on the $G(\Q)$-equivalence class of $v$ in $V(\Z)$. Then
\begin{equation}\label{eq:weights}
\sum_{i=1}^k \frac{w(v_i)}{\#\Stab_{G(\Z)}(v_i)}=\Bigl(\sum_{i=1}^k \frac{\#\Stab_{G(\Q)}(v)}{\#\Stab_{G(\Z)}(v_i)}\Bigr)^{-1}\sum_{i=1}^k \frac{1}{\#\Stab_{G(\Z)}(v_i)}=\frac{1}{\#\Stab_{G(\Q)}(v)}.
\end{equation}
When $\Stab_{G(\Q)}(v)$ is trivial, which happens for all but negligibly
many $v\in V(\Z)$ by \cite[Proposition~14]{B}, \eqref{eq:weights} simplifies to
\begin{equation}\label{eq:weightssimplify}
\sum_{i=1}^k w(v_i)=1.
\end{equation}
Since the size of $\Stab_{G(\Q)}(v)$ is bounded above by $2^{2g}$, \eqref{eq:weightssimplify} always holds up to an absolutely bounded factor. Therefore, the right hand side of
\eqref{eqtheoremimpone} also counts the number of $G(\Q)$-equivalence
classes of elements in $V(\Z)$ that are locally soluble, have
invariant binary forms in $F$, and have height bounded by $X$.
\end{proof}

\medskip

In order to apply Theorem \ref{thinfcong} to bound $N_w(V(\Z)^{(m,r)};X),$ we need to know that $w$ is defined by congruence conditions.
The proof that $w$ is indeed a product $\prod_p w_p$ of local weight
functions is identical to the proof of \cite[Proposition~3.6]{BS}. Therefore, to bound $N_w(V(\Z)^{(m,r)};X),$ it remains to compute $c_{m,r}$ and the $p$-adic integrals $\int w_p(v)dv$. We fix left-invariant top differentials $d\tau,d\mu$ on $G$ and $\A_\Z^{n+1}$ defined over $\Z$ and denote by $\tau_\infty,\tau_p,\mu_\infty,\mu_p$ the induced measures on $G(\R),G(\Q_p),\R^{n+1},\Q_p^{n+1}$ respectively. We normalize $d\mu$ such that $\mu_\infty$ is the usual Euclidean measure on $\R^{n+1}$ and $\mu_p(\Z_p^{n+1})=1$ for all primes~$p$. Then, we have the following results:
\begin{eqnarray*}
c_{m,r}X^{n+1}&=&\frac{|\J|\tau_\infty(G(\Z)\backslash G(\R))}{\#J[2](\R)}\mu_\infty(\{f\in I(m)| H(f) < X\});\\
\int_{v\in V(\Z_p)}w_p(v)dv&=&|\J|_p\tau_p(G(\Z_p))\mu_p(F_p)\frac{\#(J^1(\Q_p)/2J(\Q_p))}{\#J[2](\Q_p)};
\end{eqnarray*}
here $\J$ is a nonzero rational constant; $J$ denotes the Jacobian of any
hyperelliptic curve defined by $z^2 = f(x,y)$ where $f(x,y)\in F\cap I(m)$; and $F_p$ is the $p$-adic closure of $F.$
The first equation is proved in \cite[\S 4.4]{B}. The second equation follows from the identical computation as in \cite[\S 4.5]{SW}.

For every place $\nu$ of $\Q$, we let $a_\nu$ denote the following quotient:
$$a_\nu = \frac{\#(J^1(\Q_\nu)/2J(\Q_\nu))}{\#J[2](\Q_\nu)}.$$
Because of the assumption  that $J^1(\Q_\nu)\neq\varnothing$, this quotient depends only on $\nu$, $g$. Indeed, it is equal to $2^{-g}$ for $\nu=\infty$, \,$2^g$ for $\nu=2$, and $1$ for all other primes (see, e.g., \cite[Lemma~12.3]{BG}). The $a_\nu$'s satisfy the product formula $\prod_\nu a_\nu = 1.$

We now combine Theorem \ref{thinfcong}, Theorem \ref{thmisimp}, and the product formula $\prod_\nu |\J|_\nu=1$ to obtain:

\begin{theorem}\label{thm:main} Let $F$ be a large subset of $F_0$ contained in $F_1$. Moreover, suppose that the coefficients of every $f(x,y)$ in $F$ lie in $16^n\Z$. Then
\begin{equation}\label{eq:locsolinV}
\sum_{\substack{C\in F\\H(C)< X}}\#\Sel_2(J^1) \leq \sum_{m=0}^{n/2}\tau(G)\mu_\infty(\{f\in I(m)| H(f) < X\})\prod_p \mu_p(F_p) + o(X^{n+1}),
\end{equation}
where $\tau(G)=2$ denotes the Tamagawa number of $G$.
\end{theorem}

\section{Proofs of main theorems}\label{mainproofs}

All the results stated in the introduction, starting with Theorem \ref{count}, hold even if for each $g\geq 1$
we range over any large congruence
family of hyperelliptic curves $C$ over $\Q$ of genus~$g$ for which $\Div^1(C)$ is locally soluble.
(See Definition \ref{def:large} for the definition of ``large''.)

We prove Theorems~\ref{count} and~\ref{countk} in this generality.

\vspace{.1in}

\noindent \textbf{Proof of Theorem \ref{count}:} Let $F$ be a large
family of hyperelliptic curves with locally soluble $\Div^1$. Since
Condition 2 in Definition \ref{def:large} is a mod $p$ condition, the Ekedahl sieve as in \cite[Theorem 3.3]{geosieve} can be applied to obtain the following tail estimate.

\begin{proposition}\label{propunif}
  Let $F_p$ denote the $p$-adic closure of $F$ in $\Z_p^{n+1}$. For any $M>0$, we have
$$
\displaystyle\#\bigcup_{p>M}\{f\in I(m)| f\notin F_p, H(f)<X\}=O(X^{n+1}/M) + O(X^n),
$$
where the implied constant is independent of $X$ and $M$.
\end{proposition}

Then by a sifting argument just as in \cite[\S2.7]{BS}, we have
\begin{equation}\label{eq:numinv}
\sum_{\substack{C\in F\\H(C)< X}}1 = \sum_{m=0}^{n/2}\mu_\infty(\{f\in I(m)| H(f) < X\})\prod_p \mu_p(F_p) + o(X^{n+1}).
\end{equation}
Dividing \eqref{eq:locsolinV} by \eqref{eq:numinv} gives Theorem~\ref{count}.  \hfill$\Box$

\vspace{.1in}

\noindent \textbf{Proof of Theorem \ref{countk}:} Let $F$ be a large
family of hyperelliptic curves with locally soluble $\Div^1$. Let $k>0$ be an odd integer. Recall that the 2-Selmer set of order $k$ is defined to be the subset of elements of $\Sel_2(J^1)$ that locally everywhere come from points in $J^1(\Q_\nu)$ of the form $d_1-\frac{k-1}{2}d$ where $d_1$ is an effective divisor of degree $k$ and $d$ is the hyperelliptic class. To obtain the average size of this 2-Selmer set of order $k$, we need to perform a further sieve from the whole 2-Selmer set to this subset. Let $\varphi_\nu\leq 1$ denote the local sieving factor at a place $\nu$ of $\Q$. Then to prove that the average size of the 2-Selmer set of order $k$ is less than 2, it suffices to show that $\varphi_\nu<1$ for some $\nu$.

We use the archimedean place. Suppose that $f(x,y)$ is a degree $2g+2$ binary form having $2m$ real linear factors with $m>0$ and let $C$ be its associated hyperelliptic curve. Then $C(\R)$ has $m$ connected components and $J(\R)/2J(\R)$ has size $2^{m-1}$. Let $\sigma$ denote complex conjugation. Then for any $P\in C(\C)$ with $x$-coordinate $t\in\C^\times$, we have that $(t-\be)(\gl{t}{\sigma}-\be)=N_{\C/\R}(t-\be)\in\R^{\times2}$ for any $\be\in\R$. Hence the descent ``$x-T$'' map sends the class of $(P)+(\gl{P}{\sigma})-d$ to 1 in $L^\times/L^{\times2}\R$ where $L$ denotes the \'{e}tale algebra of rank $n$ associated to $f(x,y)$. Thus $(P)+(\gl{P}{\sigma})-d\in 2J(\R).$ Therefore, the image of $(\Sym^k(C))(\R)$ in $J^1(\R)/2J(\R)$ is equal to the image of $\Sym^k(C(\R))$ in $J^1(\R)/2J(\R).$ Since $m$ is positive, $C$ has a rational Weierstrass point over $\R$. Hence if $P\in C(\R)$, then $2(P)-d\in 2J(\R).$ Since $C(\R)$ has $m$ connected components, we see that the image of $\Sym^k(C(\R))$ in $J^1(\R)/2J(\R)$ has size at most
$$S_m(k)=\binom{m}{1}+\binom{m}{3}+\cdots+\binom{m}{k}.$$
There is a positive proportion of hyperelliptic curves $C:z^2=f(x,y)$ in $F$ such that $f(x,y)$ splits completely over $\R$. For any odd integer $k<g$, we have $S_{g+1}(k)<2^g=|J^1(\R)/2J(\R)|$. Therefore, $\varphi_\infty<1.$

Consider now the second statement that the average size of the 2-Selmer set of order $k$ goes to 0 as $g$ approaches $\infty$. We use the archimedean place again. Suppose that $f(x,y)$ is a degree $n=2g+2$ binary form having $2m$ real linear factors and let $C$ be its associated hyperelliptic curve. For a fixed odd integer $k>0$, we have
\begin{equation}\label{eq:limit}
\lim_{m\rightarrow\infty}\frac{S_m(k)}{|J^1(\R)/2J(\R)|}=\lim_{m\rightarrow\infty}\frac{S_m(k)}{2^{m-1}}=0.
\end{equation}
On the other hand, \cite[Theorem~1.2]{DPSZ} states that the density of real polynomials of degree $n$ having fewer than $\log n/\log\log n$ real roots is $O(n^{-b+o(1)})$ for some $b>0$. Therefore, the result now follows from this and \eqref{eq:limit}. \hfill$\Box$

\vspace{.2in}

Our approach to Theorem \ref{counttwocover} (which in turn implies Theorems~\ref{count2} and \ref{count3}), using a result of Dokchitser and Dokchitser (Appendix A), does not work in the generality of large families, but does work for ``\fun'' families as defined below.

\begin{defn}\label{def:fun} \emph{A subset $F$ of the set $F_0$ of all integral binary forms of degree $n$ is} \fun \emph{ if the following conditions are satisfied:
\begin{enumerate}
\item It is defined by congruence conditions;
\item For large enough primes $p$, the $p$-adic closure of $F$ contains all binary forms $f(x,y)$ of degree $n$ over $\Z_p$ such that the hyperelliptic curve $z^2=f(x,y)$ has a $\Q_p$-rational point.
\end{enumerate}
We say that a family of hyperelliptic curves $z^2=f(x,y)$ is} \fun \emph{ if the set of binary forms $f(x,y)$ appearing is \fun.}
\end{defn}
%

To prove Theorem \ref{counttwocover}
where we range over any admissible family of hyperelliptic curves over~$\Q$ of genus $g\geq 1$ with locally soluble $\Div^1$,
we note that the result of Dokchitser and Dokchitser holds for \fun {}
families (Theorem~\ref{appmain}). The rest of the proof is identical to that given in the introduction.

We conclude by giving a version of Theorem \ref{count2}
in the most general setting that our methods allow.

\begin{theorem}\label{thm:count2end}
Suppose $F$ is a  large congruence family of integral binary forms of degree $n=2g+2$
 for which there exist two primes $p,q$ neither of which is a quadratic residue modulo the other such that the following conditions hold for a positive proportion of $f(x,y)$ in $F$:
\begin{enumerate}
\item[$1.$] The four integral binary forms $f(x,y)$, $pf(x,y)$, $qf(x,y)$, $pqf(x,y)$ all lie inside $F$ and the hyperelliptic curves have points over $\Q_p$ and $\Q_q$.
\item[$2.$] If $J$ denotes the Jacobian of the hyperelliptic curve $z^2=f(x,y)$, then $J$ has split semistable reduction of toric dimension $1$ at $p$ and good reduction at~$q$.
\end{enumerate}
Then for a positive proportion of binary forms $f(x,y)$ in $F$, the corresponding hyperelliptic curve $C:z^2=f(x,y)$ has no points over any odd degree extension of $\Q$ $($i.e., the variety $J^1$ has no rational points$)$, and moreover the $2$-Selmer set $\Sel_2(J^1)$ is empty.
\end{theorem}


\appendix
\renewcommand\thesection{Appendix \Alph{section}:}

\begin{center}
\section[\hspace{.72in}A positive proportion of hyperelliptic curves have odd/even 2-Selmer rank]{A positive proportion of hyperelliptic curves have odd/even 2-Selmer rank}

\renewcommand\thesection{\Alph{section}}

{\bf by Tim and Vladimir Dokchitser}
\end{center}

In this appendix we show that both odd and even 2-Selmer ranks occur a
positive proportion of the time among hyperelliptic curves of a given
genus.

For an abelian variety $A$ defined over a number field $K$, write
$\rk_2(A/K)=\dim_{\F_2}\Sel_2(A/K)$ for the 2-Selmer rank, and
$\rk_{2^\infty}(A/K)$ for the $2^\infty$-Selmer rank%
\footnote{
  Mordell--Weil rank + number of copies of $\Q_2/\Z_2$ in $\Sha_{A/K}$;
  if $\Sha$ is finite, this is just the Mordell--Weil rank.}.
We will say `rank of a curve' meaning `rank of its Jacobian'.

\begin{theoremx}
\label{mainQ}
The proportion of both odd and even $2^\infty$-Selmer ranks in the family of hyperelliptic
curves over $\Q$,
$$
  \z^2 = a_n x^n+ a_{n-1}x^{n-1}+\cdots +a_0\qquad (n\ge 3),
$$
ordered by height as in $(\ref{heightdef})$ is at least $2^{-4n-4}$. In particular, assuming finiteness of
the $2$-part of $\Sha$, at least these proportions of curves have Jacobians of odd and of even
Mordell--Weil rank.
\end{theoremx}

\begin{theoremx}
\label{appmain}
Let $K$ be a number field with ring of integers $\O$.
Fix $n\ge 3$.  Consider the family of all hyperelliptic curves
$$
  \z^2 = a_n x^n + a_{n-1}x^{n-1} + \cdots + a_0, \quad\qquad a_i\in\O,
$$
or any other ``\fun'' family (see Definition {\em \ref{def:fun}}). Then a positive proportion
of the hyperelliptic curves in the family, when ordered by height, have even $2$-Selmer rank and
a positive proportion have odd $2$-Selmer rank.
The same conclusion holds for the $2^\infty$-Selmer rank.
\end{theoremx}

The proofs resemble that of \cite[\S4.1]{BhSh} for elliptic curves over $\Q$.
Recall that the conjecture of Birch and Swinnerton-Dyer implies, in particular,
that the parity of the rank of an elliptic curve $E$ is determined by whether
its root number --- that is, the sign of the functional equation of the $L$-function $L(E, s)$
of $E$ --- is $+1$ or $-1$.
The proof in \cite{BhSh} uses that twisting by $-1$ does not affect the height of the curve
but often changes the root number, and that the parity of the Selmer rank is (unconjecturally)
compatible with the root number.

This compatibility is not known for hyperelliptic curves
(but see the forthcoming work of A.~Morgan for 2-Selmer ranks for quadratic twists).
Instead, we tweak the argument to use Brauer relations in biquadratic extensions,
where it is known in enough cases.
To illustrate the method, consider an elliptic curve $E/\Q$
with split multiplicative reduction at 2.
Then it has root number $-1$ over $F=\Q(i,\sqrt 2)$, since
the unique place above 2 in $F$ contributes $-1$, while
every other rational place splits into an even number of places in $F$ and so
contributes $+1$. In other words, the sum of the Mordell--Weil
ranks for the four quadratic twists
$$
  \rk(E/F) = \rk(E/\Q) + \rk(E_{-1}/\Q) + \rk(E_2/\Q) + \rk(E_{-2}/\Q) \eqno{(*)}
$$
should be odd, and so both odd and even rank should occur among
the 4 twists.
The point is that for the $2^\infty$-Selmer rank, the parity in $(*)$ can
be computed unconditionally, using a Brauer relation
in $\Gal(F/\Q)\cong \Cy_2\times \Cy_2$.
Moreover, this works for general abelian varieties and over a general
number field $K$, replacing $\Q(i,\sqrt 2)$ by a suitable biquadratic
extension of $K$. The fact that most of the decomposition groups are cyclic
allows us to avoid all the hard local computations and restrictions on the
reduction types, and varying the curve in the family gives the required
positive proportions.

The exact result we will use is:

\begin{theoremx}
\label{corrkodd}
Let $F=K(\sqrt\alpha,\sqrt\beta)$ be a biquadratic extension of number
fields. Suppose that some prime $\p_0$ of $K$ has a unique prime
above it in~$F$. Let $\X/K$ be a curve with Jacobian $J$,
such that
\begin{enumerate}
\item[$1.$]
$\X(K_{\p_0})\ne\varnothing$ and
$J$ has split semistable reduction of toric dimension~$1$ at $\p_0;$
\item[$2.$]
$\X(K_{\p})\ne\varnothing$ and $J$ has good reduction
at $\p$ for every $\p\ne\p_0$ that has a unique prime above it in $F/K$.
\end{enumerate}
Then
$$
  \rk_{2^\infty}(J/K) + \rk_{2^\infty}(J_{\alpha}/K) + \rk_{2^\infty}(J_{\beta}/K)
    + \rk_{2^\infty}(J_{\alpha\beta}/K) \equiv 1 \mod 2.
$$
If, in addition, $\X_\alpha(K_{\p})$, $\X_\beta(K_{\p})$ and $\X_{\alpha\beta}(K_{\p})$
are non-empty for all primes $\p$ of $K$ that have a unique prime above them in $F$,
then the same conclusion holds for the $2$-Selmer rank as well.
\end{theoremx}



Postponing the proof of this theorem,
we first explain how it implies Theorems~\ref{mainQ}
and~\ref{appmain}.

\subsection*{Proof of Theorem \ref{mainQ}}

For Theorem \ref{mainQ}, it suffices to prove the following:

\begin{propositionx}
Consider a squarefree polynomial $f(x)\in\Q[x]$,
$$
  f(x) = a_n x^n + a_{n-1}x^{n-1} + \cdots + a_0 \qquad \qquad (n\ge 3,\>\> n=\text{$2g+1$ or $2g+2$}),
$$
whose coefficients satisfy $a_2\equiv 1\mod 8$, $a_{2g+1}\equiv 4\mod 8$
and all other $a_i\equiv 0\mod 8$.
%
Then among the four hyperelliptic curves
$$
  \z^2\!=\!f(x), \quad\quad \z^2\!=\!-f(x), \qquad \z^2\!=\!2f(x), \qquad \z^2\!=\!-2f(x)
$$
at least one has even and at least one has odd $2^\infty$-Selmer rank.
%
%
\end{propositionx}

\begin{proof}
Replacing $y\mapsto 2y+x$ in $C: y^2=f(x)$ and dividing the equation by~4
yields a curve with reduction
$$
  \bar C/\F_2: y^2+xy = x^{2g+1}.
$$
This equation has a split node at $(0,0)$ and no other singularities, so $\Jac(C)$ has split
semistable reduction at 2 of toric dimension 1. Hensel lifting the non-singular point at
$\infty$ on $\bar C$ we find that $C(\Q_2)\ne\varnothing$.
Now apply Theorem \ref{corrkodd} with $K=\Q$, $F=\Q(i,\sqrt 2)$ and $\p_0=2$. (Note that
all odd primes split in $F/\Q$, and that $\Jac(C_\alpha)=(\Jac(C))_\alpha$.)
\end{proof}


\subsection*{Proof of Theorem \ref{appmain}}

%

\begin{lemmax}
\label{toricexistence}
Let $K$ be a finite extension of $\Q_p$ $(p$ odd$),$ with residue field~$\F_q$.
Take a hyperelliptic curve
$$
  \X: \z^2 = a_nx^n+ a_{n-1}x^{n-1} + \cdots+a_0, \qquad \qquad a_i\in\O_K,
$$
and let $f(x)\in\F_q[x]$ be the reduction of the right-hand side.
\begin{enumerate}
\item[$1.$]
If $f$ is squarefree of degree $n$ and has an $\F_q$-rational root,
then $\Jac(\X)$ has good reduction, and $\X_\alpha(K)\ne\varnothing$
for every $\alpha\in K^\times$.
\item[$2.$]
If $f(x)=(x-a)^2h(x)$ for some $a\in\F_q$ and some squarefree polynomial $h(x)$ of degree $n-2$ that possesses an $\F_q$-rational root and satisfies $h(a)\in\F_q^{\times^2}$, then $\Jac(\X)$ has split semistable reduction of toric dimension $1$, and $\X_\alpha(K)\ne\varnothing$ for every $\alpha\in K^\times$.
\item[$3.$]
If $f(x)$ is not of the form $\lambda h(x)^2$, $\lambda\in\F_q$, and $q>4n^2$,
then $C(K)\ne\varnothing$.
\end{enumerate}
\end{lemmax}

\begin{proof}
In the first case, $\X$ has good reduction, and therefore so does $\Jac(\X)$.
In the second case, $\X$ has one split node and no other singular points, and so its Jacobian
has split semistable reduction of toric dimension 1.
In both cases, $f(x)$ has a simple root $b\in\F_q$, by assumption. Lifting it by Hensel's
lemma, we get a point $(B,0)$ on $\X/K$. This point gives a $K$-rational point on every
quadratic twist of~$\X$.

For (3), this is the argument in \cite[Lemma~15]{PS}: write $f(x)=l(x)h(x)^2$
with $l$ and $h$ coprime and $l$ non-constant and squarefree. By the Weil conjectures, the curve
$y^2=l(x)$ has at least $q+1-n\sqrt{q} > n$ rational points over $\F_q$. So there is at least
one whose $x$-coordinate is not a root of $f$. It is non-singular on $y^2=f(x)$, and
by Hensel's lemma it lifts to a point in $C(K)$.
\end{proof}

\begin{proof}[Proof of Theorem \ref{appmain}]
Write $\O$ for the ring of integers of $K$, and $\F_\p$ for the residue field at $\p$.

Suppose we are given an \fun {} family $\cF$ of hyperelliptic curves.
In other words,
for every prime $\p$ the curves lie in some open set $\cF_\p$ of $\p$-adic curves $C/\O_\p$,
defined by congruence conditions modulo $\p^{m_\p}$,
and outside a finite set of primes $\Sigma$ of $\O$ these sets include all curves with
$C(\O_\p)\ne\varnothing$. Enlarge $\Sigma$ to include all primes $\p|2$,
with $m_\p$ chosen so that units of the form $1+p^{m_\p}$ are squares in $\O_\p$,
and all primes of norm $\le 4n^2$.

Take a prime $\p_0\notin \Sigma$. Pick $\alpha,\beta\in\O$ with
$\alpha\equiv\beta\equiv 1\mod\prod_{\p\in \Sigma}\p^{m_\p}$, and
such that $\alpha$ has valuation 1 at $\p_0$ and $\beta$ is a non-square unit modulo $\p_0$.
Then $\p_0$ ramifies in $K(\sqrt\alpha)$ and is inert in $K(\sqrt\beta)$, so
$F=K(\sqrt\alpha,\sqrt\beta)$ is a biquadratic extension with a unique
prime above $\p_0$. There is a finite set of primes $U$ of $K$ that have a unique prime above
them in $F$, and $U\cap\Sigma=\varnothing$. (The set is finite since such primes must
ramify in $F/K$.)

Within our family $\cF$ consider those curves $\X:y^2=f(x)$ whose reductions are as in
Lemma \ref{toricexistence}(2) at $\p_0$, as in Lemma \ref{toricexistence}(1) at all
$\p\in U\setminus\{\p_0\}$, and such that $f$ mod $\p$ is not a unit times the square of
a polynomial at any $\p\notin\Sigma\cup U$.
(This is a positive proportion of curves in $\cF$ by \cite{PS2}.)
For each such curve $\X$,
Theorem \ref{corrkodd} implies that both odd and even 2-Selmer ranks occur
among the twists of $\Jac(\X)$ by $1, \alpha, \beta$ and $\alpha\beta$,
in other words the Jacobians of $\X, \X_\alpha, \X_\beta$ and $\X_{\alpha\beta}$.
Note that these twists are in $\cF$, since for $\p\in\Sigma$ this twisting does not change
the class modulo $\p^{m_\p}$, while for $\p\not\in\Sigma$ these twists are all locally
soluble by Lemma \ref{toricexistence}(3).

Because quadratic twists by $\alpha, \beta$ and $\alpha\beta$
only change the height by at most $N_{K/\Q}(\alpha\beta)^n$, we get the
asserted positive proportion.
\end{proof}

\subsection*{Proof of Theorem \ref{corrkodd}}

We refer the reader to \cite[\S2]{tamroot} for the theory of Brauer relations
and their regulator constants.

\begin{notation}{\em
Let $F/K$ be a Galois extension of number fields with Galois group $G$,
and $A/K$ an abelian variety. Fix a global invariant exterior
form $\omega$ on $A/K$. For $K\subset L\subset F$ and a prime $p$, we write
\par\smallskip\noindent
\begin{tabular}{lll}
&$\Sha^{[p]}_{\scriptscriptstyle A/L}$ & $p$-primary part of $\Sha_{A/L}$ modulo divisible elements
    $($a finite abelian $p$-group$)$.\cr
&$ \tam AL$              & $\prod c_v |{\omega}/{\neron{v}}|_v$, where
   the product is taken over all primes of $L$,
   $c_v$ is the Tamagawa number \cr &&of $A/L$ at $v$, $\neron{v}$ the N\'eron
   exterior form and $|\cdot|_v$ the normalised absolute value at $v$.\\[2pt]
\noalign{\noindent In the theorem below we write}\\[-6pt]
&$\cS$& the set of self-dual irreducible $\Q_p G$-representations.\\
&$\Theta$&$=\sum n_i H_i$ a Brauer relation in $G$ $($i.e. $\sum_i n_i\Ind_{H_i}^G\triv=0)$.\\
&$\RC(\Theta,\rho)$&the regulator constant
$\prod_i\det\bigl(\tfrac{1}{|H_i|}\llara|\rho^{H_i}\bigr)^{n_i}
  \in \Q_p^*/\Q_p^{*2}$, \cr
&&where $\llara$ is some non-degenerate $G$-invariant pairing on $\rho$.
\end{tabular} }
\end{notation}

\noindent
Finally, as in \cite{selfduality} we let%
\footnote{\cite{selfduality} also includes
  representations of the form $T\oplus T^*$ for some irreducible
  $T\not\cong T^*$ ($T^*$ is the contragredient of $T$), but these have
  trivial regulator constants by \cite[Cor.\ 2.25]{tamroot}.}
$$
  \cS_\Theta = \{\rho\in\cS \,|\, \ord_p\RC(\Theta,\rho)\equiv 1\mod 2\}.
$$

\begin{theoremx}
\label{imrho}
Suppose $A/K$ is a principally polarized abelian variety.
For $\rho\in\cS$ write $m_\rho$ for its multiplicity in the dual
$p^\infty$-Selmer group of~$A/F$.~Then
$$
  \sum_{\rho\in\cS_\Theta} m_\rho \equiv
  \ord_p\>\prod_i \tam A{F^{H_i}}\Sha^{[p]}_{A/F^{H_i}}
  \mod 2.
$$
\end{theoremx}

\begin{proof}
This is essentially \cite[Thm.\ 1.6]{selfduality}, except for the $\Sha^{[p]}$ term
in the right-hand side. For odd $p$, this term is a square and does not contribute to
the formula. For $p=2$, this is the formula that comes out of the proof of
\cite[Thm.\ 1.6]{selfduality}. There the main step of the proof (\cite[Thm.\ 3.1]{selfduality})
assumes that $A/K$ has a principal polarization induced by a $K$-rational divisor to
get rid of the $\Sha^{[2]}$ term coming from \cite[Thm.\ 2.2]{selfduality}.
\end{proof}

\begin{corollaryx}
\label{prrkloc}
Let $F=K(\sqrt\alpha,\sqrt\beta)$ be a biquadratic extension of number
fields. For every principally polarized abelian variety $A/K$,
$$
\eqno{(\dagger)}
\begin{array}{llllllllllll}
  \displaystyle
  \rk_{2^\infty}(A/K) + \rk_{2^\infty}(A_{\alpha}/K) + \rk_{2^\infty}(A_{\beta}/K) + \rk_{2^\infty}(A_{\alpha\beta}/K)
    \equiv \qquad \qquad   \\[5pt]
  \displaystyle
  \qquad \qquad \qquad \equiv \ord_2
  \frac{\tam A{K(\sqrt\alpha)}\>\tam A{K(\sqrt\beta)}\>\tam A{K(\sqrt{\alpha\beta})}}
       {\tam AF\>(\tam AK)^2}\\[8pt]
  \qquad \qquad \qquad + \ord_2
  \frac{|\Sha^{[2]}_{\scriptscriptstyle A/K(\sqrt\alpha)}|
        |\Sha^{[2]}_{\scriptscriptstyle A/K(\sqrt\beta)}|
        |\Sha^{[2]}_{\scriptscriptstyle A/K(\sqrt{\alpha\beta})}|}
       {|\Sha^{[2]}_{\scriptscriptstyle A/F}| |\Sha^{[2]}_{\scriptscriptstyle A/K}|^2}
   \rlap{$\qquad\mod 2$.}       \cr
\end{array}
$$
\end{corollaryx}

\begin{proof}
Write ${1}, \Cy_2^a, \Cy_2^b, \Cy_2^c$ for the proper subgroups of $G=\Gal(F/K)$,
and $\triv, \epsilon_a, \epsilon_b, \epsilon_c$ for its 1-dimensional
representations (so $\C[G/\Cy_2^\bullet]\iso\triv\oplus\epsilon^\bullet$
for $\bullet=a,b,c$).
Thus the four $2^\infty$-Selmer ranks in question are the multiplicities of these four
representations in the dual $2^\infty$-Selmer group of $A/F$.
Now apply the theorem to the Brauer relation
\begin{equation}\label{brtheta}
  \Theta = \{1\} - \Cy_2^a - \Cy_2^b - \Cy_2^c + 2G.
\end{equation}
Its regulator constants are (see \cite[2.3 and 2.14]{tamroot})
$$
  \RC_\Theta(\triv)=  \RC_\Theta(\epsilon^a)= \RC_\Theta(\epsilon^b)=\RC_\Theta(\epsilon^c)
    = 2 \>\>\in\Q^\times/\Q^{\times 2},
$$
and so $S_\Theta=\{\triv,\epsilon_a,\epsilon_b,\epsilon_c\}$ in this case.
\end{proof}

\begin{proof}[Proof of Theorem \ref{corrkodd}]
We write the two expressions in $\ord_2(...)$ on the right-hand side
of Corollary \ref{prrkloc} as a product of local terms.
The modified Tamagawa numbers $\tam JK, \tam J{K(\sqrt\alpha)}, \ldots$
are, by definition, products over primes of $K$, $K(\sqrt\alpha)$, $\ldots$,
and we group all terms by primes of $K$.
Similarly, as shown by Poonen and Stoll in \cite[\S8]{PS},
the parity of $\ord_2\Sha^{[2]}$ is a sum of local terms that are 1 or 0
depending on whether $\Pic^{g-1}(\X)$ is empty or not over the corresponding
completion, and again we group them by primes $\p$ of $K$. This results
in an expression
$$
  \rk_{2^\infty}(J/K) + \rk_{2^\infty}(J_{\alpha}/K) + \rk_{2^\infty}(J_{\beta}/K) + \rk_{2^\infty}(J_{\alpha\beta}/K)
     \equiv \sum_\p t_\p \mod 2.
$$
There are three cases to consider for $\p$:

If there are several primes $\q_i|\p$ in $F$, then the decomposition
groups of $\q_i$ are cyclic, and this forces $t_\p=0$.  This is a
general fact about Brauer relations and functions of number fields
that are products of local terms, see \cite[2.31, 2.33, 2.36(l)]{tamroot}.

If there is a unique prime $\q|\p$ in $F$, then $\X(K_\p)\ne\varnothing$
by assumption. So $\Pic^{g-1}(\X)$ is non-empty in every extension of
$K_\p$, and all the local terms for $\Sha^{[2]}$ above $\p$ vanish.
Also $J$ has semistable reduction, again by assumption, so its N\'eron
minimal model stays minimal in all extensions. The term
$|{\omega}/{\neron{v}}|_v$ always cancels in Brauer relations in this
case, see e.g. \cite[2.29]{tamroot}.  So the only contribution to
$t_\p$ comes from Tamagawa numbers.

When $\p\ne\p_0$, the Jacobian $J$ has good reduction and the Tamagawa
numbers are trivial, so $t_\p=0$. Finally, if $\p=\p_0$, then $J$ has
split semistable reduction at $\p$ of toric dimension 1.  In this
case, the Tamagawa number term at $\p$ multiplies to
$2\in\Q^\times/\Q^{\times 2}$, in other words $t_\p=1$. This follows
e.g. from \cite[3.3, 3.23]{tamroot} for the Brauer relation
\eqref{brtheta}.  This proves the claim for the $2^\infty$-Selmer
rank.

It remains to deduce the formula for $\rk_2$ from the one for
$\rk_{2^\infty}$.  The difference between $\rk_2$ and $\rk_{2^\infty}$
comes from $\Sha^{[2]}$ and the 2-torsion in the Mordell--Weil group
on $J$, $J_{\alpha}$, $J_{\beta}$, and $J_{\alpha\beta}$.  Two-torsion is
the same for all four twists, and so gives an even contribution.  As
for $\Sha^{[2]}$, the local terms that define its parity give an even
contribution at every prime of $K$ that splits in $F$, as the twists
then come in isomorphic pairs.  At the non-split primes, all four
twists have local points by assumption, and so the local terms are 0.
\end{proof}

\subsection*{Acknowledgments}

We thank Jean-Louis Colliot-Th\'el\`ene, John Cremona, Noam Elkies,
Tom Fisher, Bjorn Poonen, Peter Sarnak, Arul Shankar, and Michael
Stoll for many helpful conversations. The first and third authors were
supported by a Simons Investigator Grant and NSF grant~DMS-1001828,
and the second author was supported by NSF grant~DMS-0901102. The authors
of the appendix were supported by Royal Society University Research
Fellowships.

\end{document}